\documentclass{amsart}
\usepackage{amsmath, amsthm, amssymb, amscd, epsfig, eucal}
\usepackage{pinlabel}

\begin{document}

\newtheorem{theorem}{Theorem}[section]
\newtheorem{lemma}[theorem]{Lemma}
\newtheorem{proposition}[theorem]{Proposition}
\newtheorem{corollary}[theorem]{Corollary}
\newtheorem{conjecture}[theorem]{Conjecture}
\newtheorem{question}[theorem]{Question}
\newtheorem{problem}[theorem]{Problem}
\newtheorem*{claim}{Claim}
\newtheorem*{criterion}{Criterion}
\newtheorem*{random_folded_theorem}{Random $f$-folded Surface Theorem~\ref{thm:random_f_folded_theorem}}
\newtheorem*{double_norm_theorem}{Double Norm Theorem~\ref{theorem:double_Gromov_norm}}
\newtheorem*{traintrack_theorem}{Traintrack Rationality Theorem~\ref{theorem:traintrack_rationality}}
\theoremstyle{definition}
\newtheorem{definition}[theorem]{Definition}
\newtheorem{construction}[theorem]{Construction}
\newtheorem{notation}[theorem]{Notation}
\newtheorem{object}[theorem]{Object}
\newtheorem{operation}[theorem]{Operation}

\theoremstyle{remark}
\newtheorem{remark}[theorem]{Remark}
\newtheorem{example}[theorem]{Example}

\numberwithin{equation}{section}

\newcommand\id{\textnormal{id}}

\newcommand\Z{\mathbb Z}
\newcommand\R{\mathbb R}
\newcommand\Q{\mathbb Q}
\newcommand\C{\mathbb C}
\newcommand\CC{\mathcal C}
\newcommand\Ss{\mathcal S}
\newcommand\A{\mathbb A}
\newcommand\E{\mathcal E}
\newcommand\F{\mathcal F}
\newcommand\T{\mathbb T}
\newcommand\RR{\mathcal R}
\newcommand\Sh{\mathcal S}
\newcommand\Aa{\mathcal A}
\newcommand\TT{\mathcal T}
\def\L{\mathcal L}
\def\1{{\bf 1}}
\def\G{\sqcup}
\def\I{\mathcal I}
\def\Oh{\widehat{O}}
\def\d{\textnormal{d}}
\def\u{\underline}
\def\SS{\Sigma}
\newcommand\Cee{\textnormal{\bf C}}
\newcommand\Bee{\textnormal{\bf B}}
\newcommand\Aitch{\textnormal{\bf H}}
\newcommand\fold{\textnormal{fold}}
\newcommand\rank{\textnormal{rank}}
\newcommand\CAT{\textnormal{CAT}}
\newcommand\fix{\textnormal{fix}}
\newcommand\ab{\textnormal{ab}}
\newcommand\area{\textnormal{area}}
\newcommand\cl{\textnormal{cl}}
\newcommand\scl{\textnormal{scl}}
\newcommand{\con}[1]{C_{#1}}
\newcommand{\length}{\textnormal{length}}
\newcommand{\weight}{\textnormal{weight}}
\newcommand{\norm}{\textnormal{norm}}
\newcommand{\genus}{\textnormal{genus}}
\newcommand{\Mod}{\textnormal{Mod}}
\newcommand{\Tor}{\textnormal{Tor}}
\newcommand{\inte}{\textnormal{int}}
\newcommand{\valence}{\textnormal{valence}}
\newcommand\til{\widetilde}
\newcommand\SL{\textnormal{SL}}
\newcommand\PSL{\textnormal{PSL}}
\newcommand\Sp{\textnormal{Sp}}
\newcommand\GL{\textnormal{GL}}
\newcommand\pre{\textnormal{pre}}
\newcommand{\Teich}{\textnormal{Teich}}
\def\rot{\textnormal{rot}}
\def\area{\textnormal{area}}
\def\im{\textnormal{im}}
\def\Aut{\textnormal{Aut}}
\def\End{\textnormal{End}} 
\def\Inn{\textnormal{Inn}}
\def\Out{\textnormal{Out}}
\def\Fat{\textnormal{Fat}}

\title{Surface subgroups from linear programming}

\author{Danny Calegari}
\address{University of Chicago \\ Chicago, Ill 60637 USA}
\email{dannyc@math.uchicago.edu}
\author{Alden Walker}
\address{University of Chicago \\ Chicago, Ill 60637 USA}
\email{akwalker@math.uchicago.edu}
\date{\today}

\begin{abstract}
We show that certain classes of graphs of free groups contain surface subgroups, including
groups with positive $b_2$ obtained by doubling free groups along collections of subgroups,
and groups obtained by ``random'' ascending HNN extensions of free groups. A special
case is the HNN extension associated to the endomorphism of a rank 2 free group 
sending $a$ to $ab$ and $b$ to $ba$; this example (and the random examples) answer in
the negative well-known questions of Sapir. We further show that the unit ball in the Gromov
norm (in dimension 2) of a double of a free group along a collection of subgroups is a finite-sided
rational polyhedron, and that every rational class is virtually represented by an extremal surface
subgroup. These results are obtained by a mixture of combinatorial,
geometric, and linear programming techniques. 
\end{abstract}

\maketitle
\setcounter{tocdepth}{1}
\tableofcontents

\section{Introduction}

\subsection{Gromov's surface subgroup question}

The following well-known question is usually attributed to Gromov:

\begin{question}[Gromov]\label{question:Gromov_question}
Let $G$ be a one-ended hyperbolic group. Does $G$ contain the fundamental group of a closed
surface with $\chi<0$?
\end{question}

Hereafter we abbreviate ``fundamental group of a closed surface with $\chi<0$'' to ``surface group'',
so that this question asks whether every one-ended hyperbolic group contains a surface subgroup.
This question is wide open in general, but a positive answer is known in certain special cases, 
including:
\begin{enumerate}
\item{Coxeter groups (Gordon--Long--Reid \cite{Gordon_Long_Reid});}
\item{Graphs of free groups with cyclic edge groups and $b_2>0$ (Calegari \cite{Calegari_graph});}
\item{Fundamental groups of hyperbolic $3$-manifolds (Kahn--Markovic \cite{Kahn_Markovic});}
\item{Certain doubles of free groups (Gordon--Wilton, Kim--Wilton, Kim--Oum 
\cite{Gordon_Wilton, Kim_Wilton, Kim_Oum});}
\end{enumerate} 
(this list is not exhaustive).

The main goal of this paper is to describe how linear programming may be used to settle the
question of the existence of surface subgroups in certain graphs of free groups, either by giving
a powerful computational tool to {\em find} surface subgroups in specific groups, or by reducing
the analysis of this question in infinite families of groups to a finite (tractable) calculation.
There are many reasons why the case of graphs of free groups is critical for Gromov's question, but
we do not go into this here, taking the interest of Gromov's question in this subclass of groups
to be self-evident.

\subsection{Statement of results}

We are able to prove the existence of surface subgroups in the following groups:
\begin{enumerate}
\item{A group $G$ with $b_2>0$ obtained by doubling a free group $F$ along a finite collection
of finitely generated subgroups $F_i$;}
\item{A group $G$ obtained as an HNN extension $F *_\phi$ where $F$ is a free group of fixed rank
and $\phi$ is a {\em random} endomorphism;}
\item{``Sapir's group'' $C=F*_\phi$ for $F=\langle a,b\rangle$ and $\phi:a \to ab,b \to ba$.}
\end{enumerate}
The sense in which this constitutes a significant advance over the results and methods in
\cite{Calegari_graph,Kim_Wilton,Kim_Oum} is that the edge groups 
are free groups of {\em arbitrary rank}, whereas in the cited papers the edge groups
were required to be cyclic.

Bullet (1) above is implied by a stronger result about the Gromov norm on $H_2$ of the double of
$F$ along the $F_i$, which we discuss in \S~\ref{subsection:Gromov}. 
Bullet (3) is reasonably self-explanatory. A precise statement of bullet (2) is:

\begin{random_folded_theorem}
Let $k\ge 2$ be fixed, and let $F$ be a free group of rank $k$.
Let $\phi$ be a random endomorphism of $F$ of length $n$. Then
the probability that $F*_\phi$ contains an essential surface subgroup is at least $1-O(C^{-n^c})$
for some $C>1$ and $c>0$.
\end{random_folded_theorem}

Here a random endomorphism of length $n$ is one that takes the generators to reduced words of
length $n$ chosen independently and randomly with the uniform distribution. We became interested
in surface subgroups of HNN extensions of free groups after discussions with Mark Sapir, who
conjectured that the subgroup $C$ does not contain a surface subgroup, and thought it was unlikely
that many HNN extensions should contain surface subgroups (other than $\Z^2$ subgroups for
endomorphisms fixing a nontrivial conjugacy class). See also \cite{Crisp_Sageev_Sapir} and
\cite{Sapir}. Therefore it seems safe to say that the
Random $f$-folded Surface Theorem is in many ways very unexpected.

\subsection{Gromov norm}\label{subsection:Gromov}

If $X$ is a $K(\pi,1)$, the Gromov norm of a class $\alpha \in H_2(X;\Q)$, denoted $\|\alpha\|$ 
is the infimum of $-2\chi(S)/n$ over all closed oriented surfaces $S$ without sphere components,
and all positive integers $n$, so that there is a map $f:S \to X$ with $f_*[S]=n\alpha$. If $G$
is a group, define the Gromov norm on $H_2(G;\Q)$ by identifying this space with $H_2(X;\Q)$
for $X$ a $K(G,1)$. The function $\|\cdot\|$ extends by continuity to $H_2(X;\R)$, where 
(despite its name) it defines a pseudo-norm in general. 

There is a relative version of Gromov norm for surfaces with boundary, and classes in
$H_2(X,Y)$ for subspaces $Y \subset X$, and when $H_2(X)=0$ this relative Gromov norm is equivalent
(up to a factor of 4) to the {\em stable commutator length} norm, as defined in \cite{Calegari_scl},
Ch.~2 (also see the start of \S~\ref{section:Gromov_norm}). 
There are equivalent definitions for pairs $G,\lbrace G_i\rbrace$
where $G$ is a group and $\lbrace G_i \rbrace$ is a family of conjugacy classes of subgroups of $G$.

In \S~\ref{section:traintrack_rationality} 
and \S~\ref{section:Gromov_norm} we develop tools to compute stable commutator length in
free groups relative to families of finitely generated subgroups, and show 
(Theorem~\ref{theorem:traintrack})
that the unit balls in the norm are finite sided rational polyhedra. By a doubling argument,
we obtain a similar theorem for Gromov norms of groups obtained from free groups by doubling along
a collection of subgroups:

\begin{double_norm_theorem}
Let $F$ be a finitely generated free group, and let $F_i$ be a finite collection of conjugacy
classes of finitely generated subgroups of $F$. Let $G$ be obtained by doubling $F$ along
the $F_i$. Then the unit ball in the Gromov norm on $H_2(G)$ is a finite sided rational
polyhedron, and each rational class is projectively represented by an extremal surface.
\end{double_norm_theorem}

Since extremal surfaces are necessarily $\pi_1$-injective, this shows that a group $G$ as
in the theorem contains a surface subgroup when $H_2(G)$ is nontrivial.

\subsection{Unity of methods}

The Double Norm Theorem and the Random $f$-folded Surface Theorem are logically independent,
and the certificates for $\pi_1$-injectivity of the surface subgroups they promise are 
quite different.
However, the surfaces in either case are constructed combinatorially from pieces obtained by
solving a rational linear programming problem; and the nature of the representation of the
surfaces by vectors, and the tools used to set up the linear programming problems, are very similar. 
Thus there is a deeper unity of methods underlying the two theorems, beyond the similarity that
both promise surface subgroups in certain graphs of free groups.

\subsection{Acknowledgments}
We would like to thank Sang-Hyun Kim, Tim Susse and Henry Wilton for helpful conversations about the
material in this paper. Danny Calegari was supported by NSF grant DMS 1005246, 
and Alden Walker was supported by NSF grant DMS 1203888.

\section{Traintrack Rationality Theorem}\label{section:traintrack_rationality}

\subsection{Graphs and traintracks}

We recall some standard definitions from the theory of graphs, traintracks and immersions,
and their connection to free groups and morphisms between them. See e.g.\/ \cite{Bestvina_Handel}
for background and more details.

We fix a free group $F$ of finite rank and a free generating set for $F$, and realize $F$
as the fundamental group of a rose $R$, identifying the generators of $F$ with the (oriented) 
edges of $R$. If $X$ is a graph, an {\em immersion} $X \to R$ is a locally injective 
simplicial map taking edges to edges. Every nontrivial conjugacy class in $F$ is represented
by an immersed loop in $R$, unique up to reparameterization of the domain (which is an oriented
circle).

\begin{definition}
Let $T$ be a graph. A {\em turn} is an ordered pair of distinct oriented edges incident to a vertex of $T$,
the first element incoming and the second outgoing. If $e_1$ is the incoming edge and $e_2$ the
outgoing edge, we denote the turn $e_1 \to e_2$.
\end{definition}
Thus, a turn is the same thing as the germ at a vertex of an oriented immersed path in $T$. 

\begin{definition}
A {\em traintrack} is a graph $T$ together with a subset of the turns at each vertex which are
called {\em admissible turns}. If $L$ is an oriented 1-manifold, an immersion $L \to T$ is
{\em admissible} if the germ of $L$ is admissible at every vertex of $T$. A {\em traintrack immersion}
is a simplicial map $T \to R$ taking edges to edges, which is locally injective on each
admissible turn.
\end{definition}
Thus if $L \to T$ is admissible, and $T \to R$ is a traintrack immersion, then $L \to R$ is an immersion.

If $X$ is a graph and we fix a simplicial map $X \to R$, we label the oriented edges of $X$ by the
generators of $F$ corresponding to the edges that they map to.
Any oriented 1-manifold mapping $L \to X$ pulls back these labels
so that each component of $L$ is labeled by a cyclic word in $F$. If $T$ is a traintrack
and $T \to R$ is a traintrack immersion and $L \to T$ is admissible, then the labels on the components
of $L$ are cyclically reduced words.

Conversely, suppose we are given a finite set $\Gamma$ of nontrivial conjugacy classes in $F$. We
let $L$ be an oriented simplicial 1-manifold with one component for each element of $\Gamma$, and each 
component labeled by the cyclically reduced word representing the given conjugacy class. There is
a unique immersion $L \to R$ compatible with the labels. We say that $L$ is {\em carried} 
by a traintrack immersion $T \to R$ if $L \to R$ factors through an admissible map $L \to T$. See
Figure~\ref{figure1}.

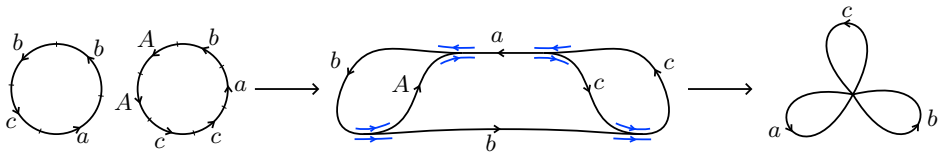
\begin{figure}[ht]
\labellist
\small\hair 2pt
\pinlabel {$b$} at 4 45
\pinlabel {$c$} at 1 9
\pinlabel {$a$} at 33 2
\pinlabel {$b$} at 40 42

\pinlabel {$A$} at 61 46
\pinlabel {$A$} at 51 18
\pinlabel {$c$} at 67 0
\pinlabel {$c$} at 92 2
\pinlabel {$a$} at 103 25
\pinlabel {$b$} at 91 47

\pinlabel {$b$} at 145 37
\pinlabel {$A$} at 174 26
\pinlabel {$b$} at 214 0
\pinlabel {$a$} at 217 47
\pinlabel {$c$} at 262 26
\pinlabel {$c$} at 293 35

\pinlabel {$a$} at 340 5
\pinlabel {$b$} at 410 11
\pinlabel {$c$} at 372 58
\endlabellist
\centering
\includegraphics[scale=0.85]{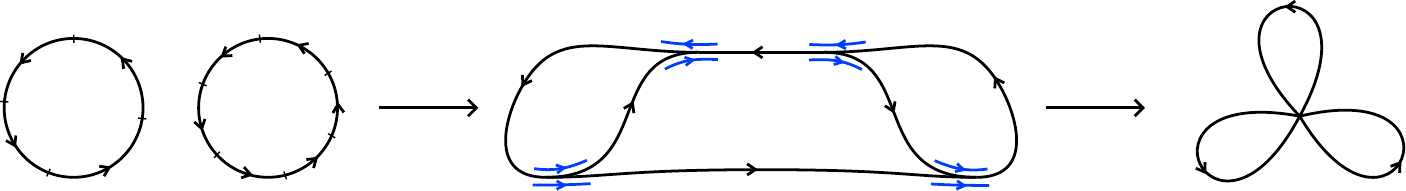}
\caption{The pair of loops $L$ maps to the rose $R$, and this 
map factors through an admissible map to the traintrack $T$, so $L$ is 
carried by $T$.}
\label{figure1}
\end{figure}

\begin{definition}
If $T$ is a traintrack, a {\em weight} $w$ is an assignment of real numbers to the admissible turns
in such a way that for each oriented edge $e$, the sum of numbers associated to turns involving $e$
at one vertex is equal to the sum at the other.
\end{definition}

The space of weights on $T$, denoted $W(T)$, is a real vector space defined over $\Q$. Weights can be non-negative,
integral, and so on. The space of non-negative weights is a convex rational cone $W^+(T)$.

A carrying map $L \to T$ determines a function from admissible turns to non-negative integers, where
the number assigned to a turn is the number of times that $L$ makes such a turn when it passes through
the given vertex. We denote this function $w(L)$.

\begin{lemma}\label{weight_is_degree}
The set of functions $w(L)$ over all carrying maps $L \to T$ is precisely the set of integer
weights in $W^+(T)$.
\end{lemma}
\begin{proof}
Each edge of $L$ contributes $1/2$ to the value of $w(L)$ on the turns at its vertices, so $w(L) \in W^+(T)$.

Conversely, let $w$ be a non-negative integer weight in $W^+(T)$. For each turn $e \to e'$ with weight
$n$, take $n$ disjoint intervals made by gluing the front half of $e$ to the back half of $e'$, and glue
these oriented intervals together (over all turns) compatibly with how they immerse in $T$ to produce $L$.
The defining property of a weight says that this gluing can be done, and $w(L)=w$.
\end{proof}

Note that $w(L)$ does not determine the topology of $L$ (i.e.\/ the number of components).
But it does determine the image of $L$ in $H_1(F)$ under $L \to R$. Thus we obtain
a homomorphism $h:W(T) \to H_1(F)$, defined over $\Q$, so that the image of $[L]$ in $H_1(R)=H_1(F)$ 
is $h(w(L))$.

\subsection{Fatgraphs and scl}

For an introduction to fatgraphs, see \cite{Penner}.

\begin{definition}
A {\em fatgraph} is a graph $X$ together with a choice of cyclic ordering of the edges incident to each
vertex. A fatgraph admits a canonical {\em fattening} to a compact oriented surface $S(X)$ in such a way
that $X$ sits inside $S(X)$ as a spine to which $S(X)$ deformation retracts. The boundary $\partial S(X)$
is an oriented 1-manifold, which comes with a canonical map $\partial S(X) \to X$ which is the
restriction of the deformation retraction, and is an immersion unless $X$ has 1-valent vertices.

A {\em fatgraph over $F$} is a fatgraph $X$ together with a simplicial map of the underlying graph $X \to R$.
It is {\em reduced} if the composition $\partial S(X) \to X \to R$ is an immersion. See Figure~\ref{figure2}.
\end{definition}

\begin{figure}[ht]
\labellist
\small\hair 2pt
\pinlabel {$a$} at 13 58
\pinlabel {$A$} at 23 41
\pinlabel {$c$} at 31 13
\pinlabel {$C$} at 36 -4
\pinlabel {$b$} at 72 -4
\pinlabel {$B$} at 78 16
\pinlabel {$A$} at 84 25
\pinlabel {$a$} at 84 43
\pinlabel {$b$} at 121 24
\pinlabel {$B$} at 122 44
\pinlabel {$c$} at 160 5
\pinlabel {$C$} at 161 23
\pinlabel {$b$} at 198 23
\pinlabel {$B$} at 199 44
\pinlabel {$a$} at 161 63
\pinlabel {$A$} at 160 44
\pinlabel {$a$} at 278 11
\pinlabel {$b$} at 340 9
\pinlabel {$c$} at 308 67
\endlabellist
\centering
\includegraphics[scale=1.0]{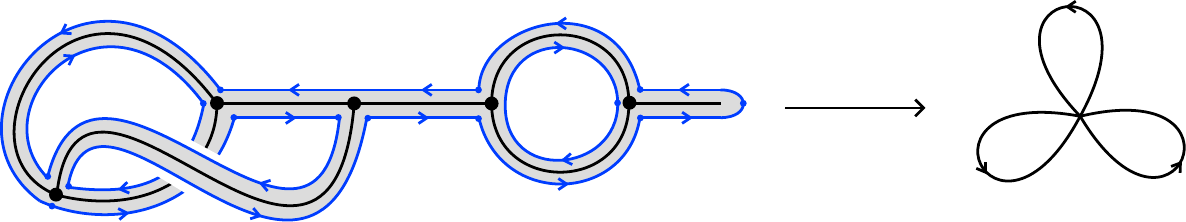}
\caption{This fatgraph map $X \to R$ is an immersion, but the fatgraph is 
\emph{not} reduced because the boundary is not reduced.}
\label{figure2}
\end{figure}
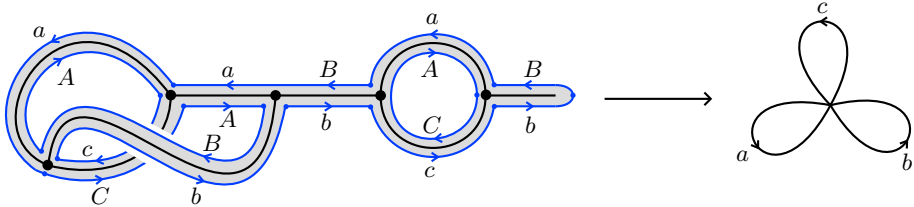

If $X \to R$ is a fatgraph over $F$ without 1-valent vertices, 
and if the underlying map of graphs $X \to R$ is an immersion,
the fatgraph is reduced. The converse is true if $X$ is 3-valent, 
but not in general otherwise. All the fatgraphs we consider in this paper will be immersed.
Moreover, throughout \S~\ref{section:traintrack_rationality} they will also be reduced. 
However we need to consider unreduced fatgraphs in \S~\ref{subsection:bounded_folding}.

Now, let $f:L \to R$ be an oriented 1-manifold mapping to $R$ by an immersion; equivalently, $L$ and $f$
are determined by the data of a collection $\Gamma$ of nontrivial conjugacy classes in $F$. 

\begin{definition}
An {\em admissible surface} for $f:L \to R$ is a compact oriented surface $S$ together with a map
$g:S \to R$ and an oriented covering map $h:\partial S \to L$ so that $f\circ h = g|\partial S$.
\end{definition}
We denote the degree of the covering map $h:\partial S \to L$ by $n(S)$. We say that ad admissible
surface $S$ is {\em efficient} if no component of $S$ is a sphere, and if every component of $S$
is geometrically incompressible; i.e.\/ if there is no essential {\em embedded} loop in $S$ mapping
to a null-homotopic loop in $R$. Any admissible surface can be replaced by an efficient one, by
throwing away sphere components and repeatedly performing compressions. Note that since by hypothesis
every component of $L$ maps to a nontrivial immersed loop in $R$, no component of $S$ is a disk, and
therefore every component of $S$ has non-positive Euler characteristic.

The following proposition is essentially due to Culler \cite{Culler} (see also \cite{Calegari_scl} \S~4.1)
and lets us reduce the study of admissible surfaces to combinatorics:

\begin{proposition}\label{proposition:surface_is_fatgraph}
Every efficient admissible surface for every oriented $f:L \to R$
is homotopic to a surface obtained by fattening a reduced fatgraph over $F$.
\end{proposition}

\begin{definition}
Let $\Gamma$ be a finite collection of conjugacy classes in $F$ whose sum is homologically trivial
(i.e.\/ represents $0$ in $H_1(F)$). The {\em stable commutator length of $\Gamma$}, denoted $\scl(\Gamma)$,
is defined to be the infimum
$$\scl(\Gamma) = \inf_S -\chi(S)/2n(S)$$
over all efficient admissible surfaces $S$ for $L$, where $f:L \to R$ represents $\Gamma$. A surface is
{\em extremal} for $\Gamma$ if equality is achieved.
\end{definition}

The main theorem of \cite{Calegari_rational} says that extremal surfaces exist for any $\Gamma$. For
more background and an introduction to the theory of stable commutator length, see \cite{Calegari_scl}
or \cite{Bavard}.

\subsection{Polygons}

Let $X$ be a reduced fatgraph over $R$ with fattening $S(X)$ and oriented boundary $\partial S(X)$. 
There is a decomposition of $S(X)$ into polygons --- canonical up to isotopy --- where all 
vertices of each polygon are vertices on $\partial S(X)$,
with one rectangle for each edge of $X$, and one $n$-gon for each
$n$-valent vertex of $X$. Each $n$-gon with $n\ge 3$ may be further decomposed into $n-2$ triangles,
without introducing new vertices; this decomposition is not canonical unless every vertex of $X$ is
at most 3-valent. Thus, we decompose $S(X)$ into two kinds of polygons: rectangles and triangles.
Note that $\chi(X) = \chi(S(X))=-\tau/2$ where $\tau$ is the number of triangles.
See Figure~\ref{figure3}.

\begin{figure}[ht]
\labellist
\small\hair 2pt
\endlabellist
\centering
\includegraphics[scale=1.0]{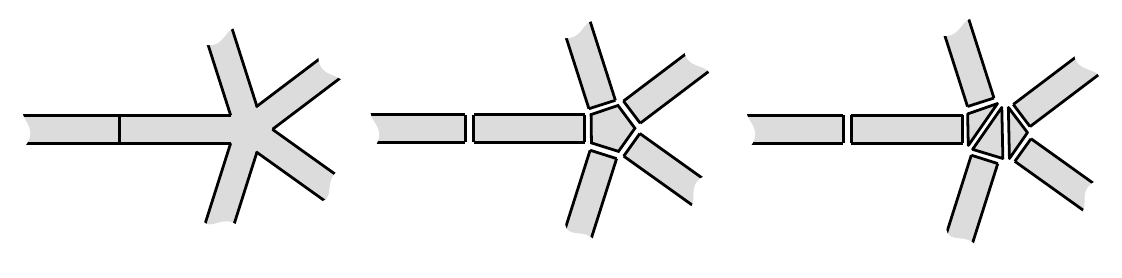}
\caption{A fatgraph $S(X)$ (left) can be cut into rectangles and polygons 
(center), and the polygons can be further cut into triangles (right).}
\label{figure3}
\end{figure}

The edges of the polygons could be {\em boundary edges}, which are edges of $\partial S(X)$, or
{\em internal edges}, which are determined by ordered pairs of vertices of $\partial S(X)$. A polygon is
determined by the cyclic list of its edges; thus, a rectangle has four edges which alternate between
boundary edges and internal edges, while a triangle has three internal edges. Note that
the edge labels on the two boundary edges of a rectangle have inverse labels. Summarizing: a rectangle
piece is determined by the data of a pair of edges of $\partial S(X)$ with inverse labels, while a
triangle is determined by the data of a cyclically ordered list of three vertices of
$\partial S(X)$. In particular, there are finitely many polygon types (at most cubic in the length
of $X$).

Now suppose that $\partial S(X)$ is carried by some immersed traintrack $T \to R$. Each rectangle
determines a pair of edges of $\partial S(X)$ with inverse labels, which are mapped to a pair
of oriented edges of $T$ with inverse labels. At each vertex, $\partial S(X)$ makes some admissible
turn in $T$; we record the information of these admissible turns at the vertices. 
Similarly, each triangle determines a cyclically ordered
list of vertices of $\partial S(X)$ which are mapped to a cyclically ordered list of admissible turns
of $T$. 

\begin{definition}
Let $T \to R$ be an immersed traintrack. A {\em triangle} over $T$ is a cyclically
ordered list of three admissible turns.
A {\em rectangle} over $T$ is a cyclically ordered list of 4 admissible turns of the form
$e_1 \to e_2$, $e_2 \to e_3$, $e_4 \to e_5$, $e_5 \to e_6$ where $e_2$ and $e_5$ 
have inverse labels. See Figure~\ref{figure4}.
\end{definition}

\begin{figure}[ht]
\labellist
\small\hair 2pt
\endlabellist
\centering
\includegraphics[scale=1.0]{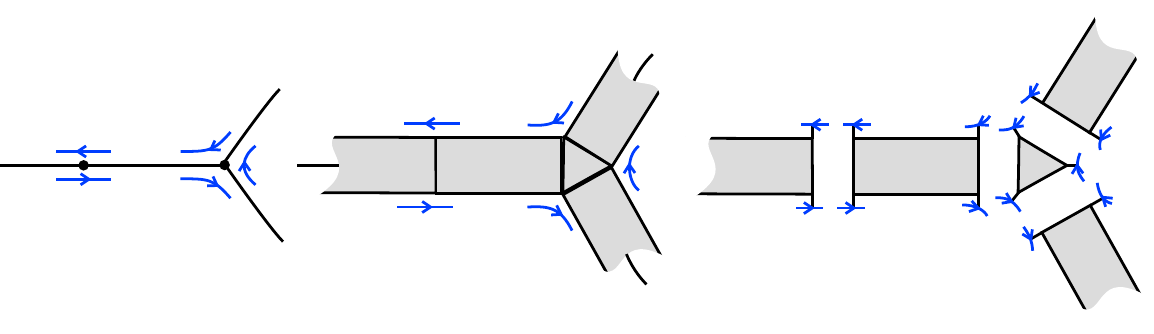}
\caption{If a fatgraph boundary $\partial S(X)$ is carried by an 
immersed traintrack $T \to R$, then each vertex of each rectangle and triangle 
is associated with an admissible turn in $T$ (blue).  As we cut $S(X)$ into 
rectangles and triangles, we record these admissible turns for each piece; 
using this information, we can reassemble the pieces into a fatgraph 
carried by $T$. }
\label{figure4}
\end{figure}

A rectangle over $T$ determines two ordered pairs $(e_2\to e_3,e_4\to e_5)$ and $(e_5\to e_6,e_1\to e_2)$
with notation as above; call these pairs the 
{\em internal edges} of the rectangle, while the edges $e_2$ and $e_5$ are the {\em boundary edges}. 
Similarly, call the three ordered pairs arising
as the boundary of a triangle over $T$ the {\em internal edges} of the triangle.

\begin{definition}
If $T \to R$ is an immersed traintrack, a {\em polygon weight} is an assignment of real numbers to
triangles and rectangles in such a way that for every unordered pair of admissible turns, the
number of times it appears as an internal edge with one ordering is the same as the number of times
it appear as an internal edge with the other ordering.
\end{definition}

The space of polygon weights on $T$, denoted $P(T)$, is a real vector space defined over $\Q$.
The space of non-negative weights is a convex rational cone $P^+(T)$. By the discussion above,
if $X$ is a reduced fatgraph over $R$ with fattening $S(X)$ and oriented boundary $\partial S(X)$
carried by $T$, then after decomposing $S(X)$ into rectangles and triangles, 
we obtain a vector $p(X)$ whose coefficients are the number of each kind of polygon over $T$
(note that $p(X)$ depends not just on $X$ but on the decomposition into triangles, although our
notation obscures this).

\begin{lemma}
Let $X$ be a reduced fatgraph over $R$ with fattening $S(X)$ and oriented boundary $\partial S(X)$
carried by $T$. Then $p(X) \in P^+(T)$.
\end{lemma}
\begin{proof}
This is just the observation that the polygons into which $S(X)$ is decomposed are glued together
in pairs along internal edges. 
\end{proof}

\begin{lemma}
There is a rational linear map $\partial:P^+(T) \to W^+(T)$ so that if $X$ is a fatgraph with $\partial S(X)$
carried by $T$, then $\partial p(X) = w(\partial S(X))$.
\end{lemma}
\begin{proof}
The map $\partial$ takes each rectangle to the vector consisting of the 4 admissible turns appearing
as vertices, each with weight $1/2$. Define $\partial$ to be zero on triangles, and
extend by linearity. This map has the desired properties.
\end{proof}
Note that $\partial$ takes integer vectors to integer vectors (though we do not use this fact).

\begin{lemma}
There is a rational linear map $-\chi:P^+(T) \to \R$ so that if $X$ is a fatgraph with $\partial S(X)$
carried by $T$, then $-\chi(p(X)) = -\chi(S(X))$.
\end{lemma}
\begin{proof}
Define $-\chi$ to be $1/2$ on every triangle, and $0$ on rectangles.
\end{proof}
Again, $-\chi$ takes integer vectors to integers.

\begin{proposition}\label{proposition:can_be_reduced}
For every non-negative integer weight $p$ in $P^+(T)$ there is some non-negative integer weight
$p'$ with $\partial p = \partial p'$ and $-\chi(p) \ge -\chi(p')$, and such that $p'=p(X)$
for some fatgraph $X \to R$ with $\partial S(X)$ carried by $T$.
\end{proposition}
\begin{proof}
An integral weight $p$ determines a collection of triangles and rectangles where the weight
of each piece determines the number of copies. Polygons can be glued together along the same
internal edge with opposite orderings; by the definition of a weight, this can be done to produce
a surface $S$ without corners. The surface $S$ might
contain some components without rectangles (i.e.\/ consisting entirely of triangles);
throw these pieces away. The surface $S$ might also contain some subsurface made entirely of 
triangles with nontrivial topology. Compress these surfaces down to disks, and triangulate
the result without introducing new vertices on the boundary. The result is a new surface which
by construction is of the form $S(X)$ for some fatgraph $X \to R$. The compression did not
affect boundary edges, so $\partial p = \partial p'$. Moreover, compression can only reduce the
number of triangles used, so $-\chi(p) \ge -\chi(p')$. This completes the proof.
\end{proof}

\subsection{Traintrack Rationality Theorem}

For $w \in W^+(T)$ rational and in the kernel of $h:W^+(T) \to H_1(F)$, 
we can define $\scl(w)$ to be the infimum of $\scl(\Gamma)/n$ for all
homologically $\Gamma$ represented by an oriented 1-manifold $L$ carried by $T$ with $w(L)=nw$ for some $n$. 
The following Traintrack Rationality Theorem is the main theorem of this section.

\begin{theorem}[Traintrack Rationality Theorem]\label{theorem:traintrack}
Let $T$ be a traintrack immersing to $R$, and let $B^+(T)$ denote the kernel of $h:W^+(T) \to H_1(F)$.
The function $\scl$ extends continuously to $B^+(T)$ in a unique way, 
where it is convex and piecewise rational linear. For any rational $w\in B^+(T)$ there is some 
homologically trivial $\Gamma$ and a fatgraph $X$ over $F$ with $\partial S(X)$ representing $\Gamma$,
in such a way that $\partial S(X)$ is carried by $T$ with $w(\partial S(X))=nw$ and
$\scl(w)=-\chi(S(X))/2n$. 
\end{theorem}
In particular, the surface $S(X)$ is extremal for $\partial S(X)$.
\begin{proof}
Define $Q(w)=P^+(w) \cap \partial^{-1}(w)$; this is a convex linear polyhedron, and is rational
if $w$ is rational. Define
$$\scl(w)=\inf_{q \in Q(w)} -\chi(q)/2$$
This is evidently convex and piecewise rational linear on $B^+(T)$. We show that it agrees with
the definition of $\scl(w)$ already given when $w$ is rational, and that there is an extremal
surface obtained from some fatgraph.

The infimum of $-\chi$ on $Q(w)$ is achieved on some nonempty subpolyhedron $E(w)$, which is convex in general,
and rational if $w$ is rational. A nonempty rational polyhedron contains a rational point, and
every rational $p \in E(w)$ can be rescaled to an integer point $np$, which is in $E(nw)$ by linearity
of the maps and $-\chi$; and by Proposition~\ref{proposition:can_be_reduced}, there is some fatgraph $X$ with
$\partial S(X)$ carried by $T$ and with $w(\partial S(X))=nw$ and $-\chi(S(X))=-\chi(np)$.

Conversely, any efficient admissible surface $S$ with $\partial S$ carried by $T$ and 
with $w(\partial S)=mw$ for some $m$ can be obtained as $S=S(X)$ for some reduced 
fatgraph $X$ over $R$ by Proposition~\ref{proposition:surface_is_fatgraph}. Then
any $p(X)$ satisfies $\partial p(X) = mw$, so $p(X) \in Q(mw)$. But then 
$$-\chi(S(X))/2m = -\chi(p(X))/2m \ge -\chi(E(w))/2$$ 
Thus $\scl(w)=-\chi(E(w))/2$, and the surface constructed from $p$ above was
extremal, as claimed.
\end{proof}

\begin{example}[Verbal traintracks]
Fix a free group $F$ of rank $k$ and a free generating set, and fix a positive integer $\ell$.
Define a traintrack $T_\ell$ whose oriented edges are the set of reduced words in $F$ of length
$\ell-1$ and whose admissible turns are reduced words of length $\ell$, which we think of as
an ordered pair of oriented edges consisting of the prefix and suffix of the given word of length
$\ell-1$.

Let $W_\ell$ denote the weight space, and $W_\ell^+$ the non-negative weights as above. There is
an involution $\epsilon$ on $W_\ell$, which takes $\sigma$ to $-\sigma^{-1}$, where $\sigma^{-1}$
denotes the inverse word to a reduced word $\sigma$. The natural inclusion $W_\ell^+ \to W_\ell$
induces a {\em surjection} $W_\ell^+ \to W_\ell/\epsilon$, and we obtain a rational linear
(pseudo)-norm on $W_\ell/\epsilon$, where the norm $\|[w]\|$ of an equivalence class $[w]$
is the infimum of the $\scl(w)$ over all $w \in W_\ell^+$ mapping to $w$. The linear functions
on $W_\ell/\epsilon$ are precisely real linear combinations of the {\em homogeneous} (big)
{\em counting quasimorphisms of length at most $\ell$} first introduced by Rhemtulla \cite{Rhemtulla}
and studied later by Brooks \cite{Brooks}, Grigorchuk \cite{Grigorchuk} and others. Thus we may
use $W_\ell/\epsilon$ to get an explicit and complete set of linear relations between the 
homogeneous counting quasimorphisms supported on words of any bounded length. 
For more details, see \cite{Calegari_Walker_sslpv1}, especially \S~4--5.
\end{example}

\section{Gromov Norm of doubles}\label{section:Gromov_norm}

We briefly introduce the Gromov norm on the homology of a space or group, and its relative variants.

\begin{definition}
Let $X$ be a topological space. The Gromov (pseudo)-norm (also called the $L_1$ norm)
of a homology class $\alpha \in H_i(X;\R)$, denoted $\|\alpha\|$,
is the infimum of $\sum |t_i|$ over all real singular $i$-cycles $\sum t_i\sigma_i$ representing 
$\alpha$. Similarly define a norm on relative classes $\alpha \in H_i(X,Y;\R)$ for a subspace
$Y \subset X$ from relative $i$-cycles.
\end{definition}

If $G$ is a group, we can define the Gromov norm on $H_*(G)$ by identifying the group
homology with $H_*(K(G,1))$. 

\begin{definition}
If $G_i$ is a family of conjugacy classes of subgroups of $G$, 
we can build a space $K$ as the mapping cylinder of $\coprod_i K(G_i,1) \to K(G,1)$, and we define
the Gromov norm on $H_*(G,\lbrace G_i\rbrace)$ by identifying group homology with
$H_*(K,\coprod_i K(G_i,1))$.
\end{definition}

In the 2-dimensional case, one has the following geometric interpretation of the Gromov norm:
\begin{proposition}\label{proposition:norm_is_surfaces}
For $\alpha \in H_2(X;\Q)$ there is a formula
$$\|\alpha\|=\inf_S -2\chi(S)/n(S)$$
where the infimum is taken over closed oriented surfaces $S$ without sphere components
for which there are maps $f:S \to X$ with $f_*[S]=n\alpha$ for some $\alpha$.

Similarly, for $\alpha \in H_2(X,Y;\Q)$ the same formula is true, where now the infimum is
taken over compact oriented surfaces $S$ without sphere or disk components for which
there are maps $f:(S,\partial S) \to (X,Y)$ with $f_*[S]=n\alpha$ for some $\alpha$.
\end{proposition}
For more details, see \cite{Gromov_bounded}; for the connection to $\scl$ in the 2-dimensional case,
see \cite{Calegari_scl}.

The following application makes no mention of traintracks in the statement, and is our main
motivation for pursuing this line of reasoning.

\begin{theorem}[Relative Gromov Norm]\label{theorem:relative_Gromov_norm}
Let $F$ be a finitely generated free group, and let $F_i$ be a finite collection of conjugacy classes of finitely 
generated subgroups of $F$. Let $H:=H_2(F,\lbrace F_i\rbrace)$ denote relative 2-dimensional
homology. Then the unit ball in the Gromov norm on $H$ is a finite sided rational polyhedron,
and each rational class is projectively represented by an extremal surface with boundary.
\end{theorem}
\begin{proof}
Let $R$ be a rose for $F$, and for each $i$ let $R_i$ be a graph without 1-valent edges
that immerses in $R$ in such a way that the image of $\pi_1(R_i)$ is conjugate to $F_i$.
Such graphs are obtained by Stallings' method of {\em folding} a set of generators for $F_i$;
see \cite{Stallings}. We let $T$ be the traintrack whose underlying graph is the disjoint union
$\cup_i R_i$, and whose admissible turns are exactly the paths in $R_i$ that do not backtrack.
We can build a space $C$ as the mapping cylinder of the immersions $\cup_i R_i \to R$; thus
$C$ retracts to $R$, and contains $\cup_i R_i$ as a subspace. For each component $T_i$ of $T$
there is a rational linear map $h:W^+(T_i) \to H_1(R_i)$, and all together these give a (surjective) 
rational linear map
$$h:W^+(T) \to \oplus H_1(R_i) = \oplus H_1(F_i)$$
Note that $\partial:H_2(F,\lbrace F_i\rbrace) \to \oplus H_1(F_i)$ is injective, and
has image equal to the kernel of $\oplus H_1(F_i) \to H_1(F)$, by the long exact sequence, 
and $H_2(F)=0$ for a free group $F$.

Any $(S,\partial S) \to (R,\cup_i R_i)$ can be homotoped and compressed until $\partial S \to \cup_i R_i$
is an immersion, which is to say it is carried by $T$. The surface $S$ can be further compressed
until we can write $S=S(X)$ for some fatgraph $X$ over $R$
compatible with $\partial S(X) \to T \to R$. Conversely, any fatgraph $X$ over $R$
with $\partial S(X)$ carried by $T$ represents a class in $H_2(F,\lbrace F_i\rbrace)$.

We can express this in terms of linear algebra as follows. 
If, as before, we denote the kernel of $h:W^+(T) \to H_1(F)$ by $B^+(T)$, and factor $h$ as
$$0 \to B^+(T) \to \oplus H_1(F_i) \to H_1(F)$$
then this sequence is exact; i.e.\/ the first map is injective on $B^+(T)$, and its image 
is exactly equal to the kernel of $\oplus H_1(F_i) \to H_1(F)$. Note that this is an exact sequence
of $\R^+$-modules, since $B^+(T)$ is merely a cone, and not a vector space. On the other hand,
since all the terms and maps are defined over $\Q$, the sequence is still exact when restricted to
the rational points in each term. Since $\partial:P^+(X) \to B^+(X)$ is surjective, and 
$\partial:H_2(F,\lbrace F_i\rbrace) \to \oplus H_1(F_i)$ is injective with image equal to the
kernel of $\oplus H_1(F_i) \to H_1(F)$, we see that we have shown that $h:P^+(T) \to H_2(F,\lbrace F_i\rbrace)$
is surjective, and for any rational $\alpha \in H_2(F,\lbrace F_i\rbrace)$ we have an equality
$$\|\alpha\|=\inf_{p \in h^{-1}(\alpha)} -2\chi(p)$$
Since $h$ is rational linear, since $P^+(T)$ is a convex rational polyhedral cone, and since
$-\chi$ is rational linear on $P^+$, it follows that the unit ball in the Gromov norm is a finite
sided rational polyhedron. Moreover, if $\alpha$ is rational, the infimum is achieved on some rational
$p$, and by Proposition~\ref{proposition:can_be_reduced} any $p$ achieving the minimum is
projectively equivalent to $p(X)$ for some $X$, in which case $S(X)$ is an extremal surface
projectively representing $\alpha$.
\end{proof}

An absolute version of Theorem~\ref{theorem:relative_Gromov_norm} may be obtained by 
{\em doubling}.  

\begin{definition}\label{def:doubling}
If $G_i$ is a family of conjugacy classes of subgroups of $G$, 
we can build a space $DK$ from two copies of the mapping cylinder $K$ of
$\coprod_i K(G_i,1) \to K(G,1)$, identified along $\coprod_i K(G_i,1)$. The
{\em double} of $G$ along the $G_i$ is the fundamental group of $DK$.
\end{definition}

Note that the double is a graph of groups, whose underlying graph has two vertices 
(corresponding to the two copies of $G$ in the double) and with one edge between the two
vertices for each $G_i$.

\begin{theorem}[Gromov Norm of Doubles]\label{theorem:double_Gromov_norm}
Let $F$ be a finitely generated free group, and let $F_i$ be a finite collection of conjugacy
classes of finitely generated subgroups of $F$. Let $G$ be obtained by doubling $F$ along
the $F_i$. Then the unit ball in the Gromov norm on $H_2(G)$ is a finite sided rational
polyhedron, and each rational class is projectively represented by an extremal surface.
\end{theorem}
\begin{proof}
This follows formally from Theorem~\ref{theorem:relative_Gromov_norm}. First of all, at the
level of homology there is a natural isomorphism $H_2(F,\lbrace F_i\rbrace) \to H_2(G)$ obtained
by identifying the $F$ factors on both sides of the double. The point is that this map is
surjective, since the $F$ factors have no absolute $H_2$ of their own (apply Mayer-Vietoris).

Any surface representing a relative class in $H_2(F,\lbrace F_i\rbrace)$ may be doubled to 
produce a closed surface representing a corresponding class in $H_2(G)$. Conversely, any surface
representing a class in $H_2(G)$ may be split into two subsurfaces on either side of the double, 
each representing the same relative class in $H_2(F,\lbrace F_i\rbrace)$. One of these subsurfaces
has $-\chi$ at most half of $-\chi$ of the big surface; doubling that subsurface produces a new
surface representing the same class in $H_2(G)$ with the same or smaller $-\chi$.

It follows that the doubling isomorphism $H_2(F,\lbrace F_i\rbrace) \to H_2(G)$ just multiplies
the norm of a class by $2$, and the double of any extremal surface for a class in 
$H_2(F,\lbrace F_i\rbrace)$ is an extremal surface for the corresponding class in $H_2(G)$.
\end{proof}

Since extremal surfaces are $\pi_1$-injective, we obtain the following corollary:

\begin{corollary}[Surface subgroups in doubles]\label{corollary:surface_subgroup_corollary}
Let $F$ be a finitely generated free group, and let $F_i$ be a finite collection of conjugacy
classes of finitely generated subgroups of $F$. Let $G$ be obtained by doubling $F$ along the
$F_i$. If $H_2(G)$ is nontrivial, then $G$ contains a surface subgroup.
\end{corollary}

For example, if $\sum \rank(F_i)>\rank(F)$ then $H_2(G)$ is nontrivial.

\begin{remark}
Theorem~\ref{theorem:double_Gromov_norm} should be compared to the case that $G=\pi_1(M)$
where $M$ is an irreducible 3-manifold. Then $\|\cdot\|$ is equal to twice the
{\em Thurston norm} on $H_2(M)$, whose unit ball Thurston famously proved is a finite-sided
rational polyhedron \cite{Thurston_norm}. There is a crucial difference between the two Theorems:
in a 3-manifold, every integral $\alpha$ is represented by a norm-minimizing embedded surface
$S$, so that $[S] =\alpha$, and therefore $\|\alpha\| \in 4\Z$, whereas for $G$ as in
Theorem~\ref{theorem:double_Gromov_norm}, the denominator of $\|\alpha\|$ can be arbitrary
for $\alpha \in H_2(G;\Z)$. This is true even when $G$ is obtained by doubling a free group
of rank 2 along a cyclic subgroup; see \cite{CW_endomorphism}. 
\end{remark}

\section{Random endomorphisms}
\subsection{HNN extensions}

Let $F$ be a finitely generated free group, and let $\phi:F \to F$ be an injective endomorphism. We obtain
an HNN extension $G:=F*_\phi$. Geometrically we can realize $F=\pi_1(R)$ for some rose $R$ as
above, and $\phi$ by a simplicial map $f:R \to R$, and build a mapping torus $K$ which is
a CW 2-complex, with one 2-cell (a square) for each generator of $F$.

There is a natural presentation 
$$G:=\langle F, t \; | \; tFt^{-1} = \phi(F)\rangle$$
and a surjection $G \to \Z$ defined by $t \to 1$ and $F \to 0$. Let $\tilde{K}$ denote the
infinite cyclic cover of $K$ associated to the kernel of this surjection; $\tilde{K}$ is
made from $\Z$ copies of $R\times I$, which we denote $K_i$ for $i \in \Z$. Denote the copy of
$R\times 1$ in $K_i$ by $\partial^+ K_i$ and the copy of $R\times 0$ in $K_i$ by $\partial^- K_i$. Then
$\tilde{K}$ is obtained by gluing each $\partial^+ K_i$ to $\partial^- K_{i+1}$ by a map
$f_i$ (which is just $f$ when we identify both domain and range in a natural way with $R$).

For any positive $n$ we denote the union $K_0 \cup_{f_0} K_1 \cup_{f_1}\cdots \cup_{f_{n-1}} K_n$ by
$K_0^n$. Observe that $K_0^n$ deformation retracts to $\partial^+ K_n$, and therefore its fundamental
group is free and isomorphic to $F$.

\subsection{$f$-fatgraphs}

Fix a rose $R$ for $F$ and a simplicial map $f:R \to R$ representing $\phi:F \to F$.

\begin{definition}\label{definition:f_fatgraph}
An $f$-fatgraph $X$ over $R$ ({\em not} assumed to be reduced or without 1-valent vertices)
is a fatgraph $g:X \to R$ together with a decomposition of
$\partial S(X)$ into submanifolds $\partial^-$ and $\partial^+$ (each a union of components)
so that there is an orientation-reversing homeomorphism $f':\partial^- \to \partial^+$
lifting $f$ (i.e.\/ satisfying $gf'=fg$ where by abuse of notation we denote the
composition $\partial S(X) \to X \to R$ by $g$).
\end{definition}

If $X$ is an $f$-fatgraph over $R$, we can replace $g:X \to R$ with a homotopic map 
of homotopy equivalent spaces $S(X) \to R\times I$, sending $\partial^-$ to $R\times 0$
and $\partial^+$ to $R\times 1$. By the defining property of an $f$-fatgraph, if we
denote by $S*_f(X)$ the closed oriented surface obtained from $S(X)$ by gluing 
$\partial^-$ to $\partial^+$ by $f'$, then the map from $S(X)$ to $K$ factors through
$S*_f(X) \to K$. Thus $f$-fatgraphs induce maps from surface groups to $F*_\phi$.
The converse is the following lemma:
\begin{lemma}
Let $S$ be a closed oriented surface, and $g:S \to K$ a map. Then $S$ and $g$
can be compressed to a surface $g':S' \to K$ which is homotopic to a map of the form
$S*_f(X) \to K$ associated to an $f$-fatgraph $X$ over $R$ with $\partial^-$ immersed in $R$.
\end{lemma}
\begin{proof}
First, throw away sphere components of $S$.
Make $g$ transverse to $R\times 0 \subset K$, so that the preimage is a system of embedded
loops $\Gamma$ in $S$. Inductively eliminate innermost complementary disks by an isotopy.
Furthermore, if some component of $\Gamma$ maps to a homotopically trivial loop in
$R\times 0$, we compress $S$ and $g$ along this loop
If $S_i$ is a component of $S$ that does not meet $\Gamma$ then
$g:S_i \to K$ factors through $S_i \to R\times I$; but any map from a closed oriented surface
to a space homotopic to a graph extends over a handlebody, so $S_i$ can be completely compressed
away. Thus we eventually arrive at $g':S' \to K$ which can be cut open along the remaining
loops $\Gamma'$ to produce a proper map $g'':S'' \to R\times I$, every boundary component of which
maps to an essential loop. Compress $S''$ further if possible. The boundary $\partial S''$
decomposes into $\partial^-$ and $\partial^+$, and the way these sit in $S'$
determines an orientation-reversing homeomorphism $\partial^- \to \partial^+$. We homotope
the map on $\partial^-$ so that it is immersed in $R$, and homotop the map on $\partial^+$
to be equal to its image under $f$. Note that if $f$ is not an immersion, neither is
the map $\partial^+ \to R$ necessarily. But $\partial^+ \to R$ factors through 
$\partial^+ \to \partial^{++} \to R$, where the first map folds some intervals into trees, and
the second map is an immersion (this is just Stallings' folding procedure applied to
$\partial^+$, together with the fact that each component maps to an essential loop in $R$).

By Proposition~\ref{proposition:surface_is_fatgraph} there is some reduced fatgraph $X$
with $\partial S(X) = \partial^- \cup \partial^{++}$; adding some trees to $X$ we obtain
a (possibly non-reduced) fatgraph $X'$ with $\partial S(X') = \partial^- \cup \partial^+$,
giving $X$ the structure of an $f$-fatgraph with $S*_f(X') \to K$ homotopic to $S' \to K$.
\end{proof}

This Lemma lets us study surfaces in $K$ (and surface subgroups mapping to $G$) combinatorially. But
actually we are interested in going in the other direction, building $f$-fatgraphs and then
using them to construct surfaces and surface subgroups in $G$.

\subsection{Stacking surfaces and fattening stacks}

If $g:X \to R$ is an immersed fatgraph over $R$ (not necessarily reduced) 
then we denote by $f(g):f(X) \to R$
the fatgraph over $R$ with the same underlying topological space as $X$, but with $f(g)=f\circ g$
and $f(X)$ subdivided so that this map takes edges to edges. If $X$ is an $f$-fatgraph, then so
is $f(X)$, and there is a natural orientation-reversing simplicial homeomorphism between
$\partial^+S(X)$ and $\partial^-S(f(X))$. Iterating this procedure, we can build a surface
$$S_n(X):=S(X) \cup S(f(X)) \cup \cdots \cup S(f^n(X))$$
The boundary labels of the $\partial S(f^n(X))$ are words obtained by applying $\phi$ 
by substitution repeatedly to the generators on the edges of $\partial S(f^i(X))$; i.e.\/ we do
{\em not} perform cancellation if these words are not reduced. See Figure~\ref{figure5}.


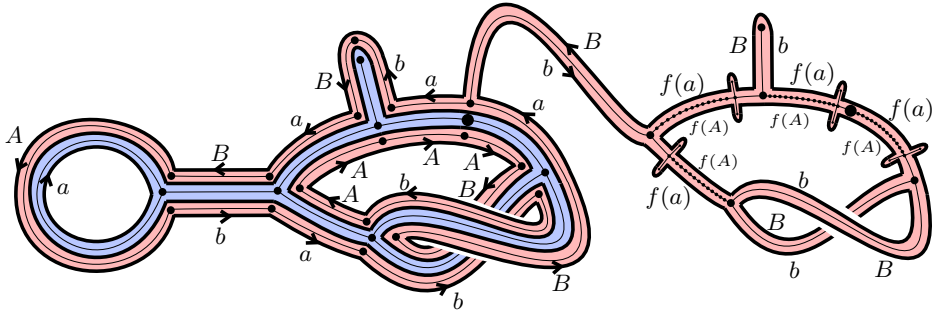
\begin{figure}[ht]
\labellist
\small\hair 2pt
 \pinlabel {$a$} at 11 24
 \pinlabel {$a$} at 93 48
 \pinlabel {$b$} at 86 52
 \pinlabel {$B$} at 69 48
 \pinlabel {$a$} at 63 39
 \pinlabel {$B$} at 46 32
 \pinlabel {$A$} at 0 36
 \pinlabel {$b$} at 46 13.5
 \pinlabel {$a$} at 65 9
 \pinlabel {$b$} at 99 -1
 \pinlabel {$b$} at 87 25.5
 \pinlabel {$A$} at 75 22.8
 \pinlabel {$A$} at 77 28.2
 \pinlabel {$A$} at 93 31
 \pinlabel {$A$} at 102 30
 \pinlabel {$B$} at 101 24
 \pinlabel {$B$} at 122 3
 \pinlabel {$a$} at 117 42.6
 
 \pinlabel {$b$} at 119 51
 \pinlabel {$f(a)$} at 146 22
 \pinlabel {$b$} at 174 6
 \pinlabel {$b$} at 176 26
 \tiny
 \pinlabel {$f(A)$} at 157 30
 \pinlabel {$f(A)$} at 155 37.5
 \pinlabel {$f(A)$} at 173 39
 \pinlabel {$f(A)$} at 189 33
 \small
 \pinlabel {$B$} at 170 16.5
 \pinlabel {$B$} at 194 6
 \pinlabel {$f(a)$} at 200 41.5                                                                                                                                 
 \pinlabel {$f(a)$} at 178 49.5                                                                                                                                 
 \pinlabel {$b$} at 171.5 56                                                                                                                                 
 \pinlabel {$B$} at 161.5 56                                                                                                                                 
 \pinlabel {$f(a)$} at 149 47                                                                                                                                 
 \pinlabel {$B$} at 129 57 
\endlabellist
\centering
\includegraphics[scale=1.7]{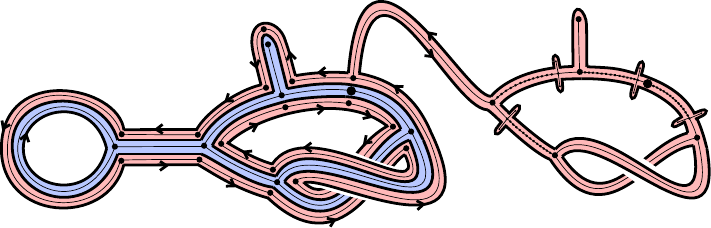}
\caption{Let $f(a) = AbbaaaaBBABabAbBA$ and $f(b) = b$.  
This figure shows a fatgraph $S(X)$ (blue) with boundary $\partial^-S(X) = a$ 
and $\partial^+S(X) = f(A)$.  The fatgraph $S(f(X))$ (red) is glued by 
identifying $\partial^+S(X)$ and $\partial^-(S(f(X))$, as shown.  Typically, 
a failure to be reduced will come from cancellation between $f(a)$ and $f(b)$.  
Here we have made $f(a)$ non-reduced for illustrative purposes.}
\label{figure5}
\end{figure}

Each $S(f^i(X))$ deformation retracts to $f^i(X)$, so there is an induced quotient map from $S_n(X)$ to
a graph $X_n$. Now, although each individual $S(f^i(X))$ is homotopy equivalent to $f^i(X)$,
it is {\em not} necessarily true that $S_n(X)$ is homotopy equivalent to $X_n$. However, this
can be guaranteed by imposing a simple condition.

\begin{lemma}\label{lemma:embedding}
Suppose that $\partial^- S(X) \to X$ is an embedding; equivalently, that no vertex of $X$ is in the
image of more than one vertex of $\partial^- S(X)$ under the deformation retraction from
$S(X)$ to $X$. Then $X_n$ admits the structure of a fatgraph in a natural way so that
$S_n(X)=S(X_n)$.
\end{lemma}
\begin{proof}
Each $S(f^i(X))$ deformation retracts to $f^i(X)$, and the tracks (i.e.\/ point preimages)
of this deformation are proper essential arcs which retract to points in the edges of $X$, 
and proper essential trees which retract to the vertices of $X$. Glue up the tracks of the 
deformation retraction for $S(f^i(X))$ to the tracks in $S(f^{i+1}(X))$ by the identification
of the boundaries; the result is a decomposition of $S_n(X)$ into {\em graphs}, in such a way that
$X_n$ is the quotient space obtained by quotienting each graph to a point. We claim 
that each such graph is a tree. Since these trees are disjointly embedded in $S_n(X)$, 
we can embed $X_n$ as a spine of $S_n(X)$ in a natural way, giving it the structure of a fatgraph
with $S(X_n)=S_n(X)$.

If $\tau$ is a track in some $S(f^i(X))$, then $\tau$ has at most one boundary point on
$\partial^-$ (by hypothesis). Define an orientation on the edges of $\tau$ in such a way that
the edges all point towards this unique boundary point on $\partial^-$ (if one exists), or towards
the unique point on $f^i(X)$ that $\tau$ deformation retracts to otherwise. See Figure~\ref{figure6}.

\begin{figure}[ht]
\labellist
\small\hair 2pt
 \pinlabel {$\partial^-$} at 28 37
 \pinlabel {$\partial^+$} at 26 7
 \pinlabel {$\partial^+$} at 57 22
\endlabellist
\centering
\includegraphics[scale=1.5]{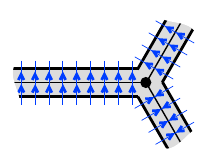}
\caption{The flow points towards $\partial^-$.}
\label{figure6}
\end{figure}

Then each graph
$T$ which is a maximal connected union of tracks in the various $f^i(X)$ gets an orientation on
its edges in such a way that each vertex has {\em at most} one outgoing edge. Thus $T$
can be canonically deformation retracted along oriented edges to a (necessarily unique) minimum,
and $T$ is a tree.
\end{proof}

Now, if $X$ is an $f$-fatgraph, we distinguish, amongst the vertices of $\partial^+$, those
which are in the image of vertices of $\partial^-$ under $f$, and call these {\em $f$-vertices}.

\begin{definition}\label{definition:f_folded}
An $f$-fatgraph $g:X \to R$ is {\em $f$-folded} if it satisfies the following conditions:
\begin{enumerate}
\item{the underlying map of graphs $X \to R$ is an immersion;}
\item{every $f$-vertex in $\partial^+$ maps to a  
2-valent vertex of $X$ under the retraction $\partial^+ \to X$;}
\item{no vertex of $X$ is in the image of more than one $f$-vertex in $\partial^+$; and}
\item{the map $\partial^- \to X$ is an embedding.}
\end{enumerate}
\end{definition}

The first condition says that the underlying map of graphs
$X \to R$ is folded in the sense of Stallings.
If $X$ has no 1-valent vertices, this implies that $X$ is reduced, but in general
$\partial S(X)$ will contain consecutive pairs of cancelling letters at 1-valent vertices
of $X$.

\begin{proposition}\label{proposition:f_folded_injective}
Suppose $f:R \to R$ is an immersion, and $X$ is $f$-folded. Then $S*_f(X) \to K$ is $\pi_1$-injective.
\end{proposition}
\begin{proof}
First, since $X\to R$ is an immersion by condition (1), and $f:R \to R$ is an immersion by
hypothesis, it follows that $f^i(X) \to R$ is an immersion for each $i$.

If $S*_f(X) \to K$ is not injective, there is some loop in the kernel. Such a loop lifts
to a loop in the infinite cyclic cover of $S*_f(X)$ which maps to $\tilde{K}$ and is contained
in the preimage of some $K_n$. But this preimage is exactly $S_n(X)$, so it suffices to show
that $S_n(X)$ maps injectively. Condition (4) implies that $S_n(X)$ is homotopy equivalent to
the fatgraph $X_n$, so it suffices to prove that $X_n \to R$ is injective, and to do this it
suffices to show that it is an immersion. But this is a local condition, and is proved by
induction on $n$, since the case $n=0$ is condition (1), and conditions (2) and (3) imply that
each vertex of $X_n$ of valence $>2$ whose restriction to $X_{n-1}$ has valence $2$ is locally
isomorphic to some vertex in $f^n(X)$, which we already saw is immersed in $R$. See Figure~\ref{figure7}.
This completes the proof.
\end{proof}

\begin{figure}[ht]
\labellist
\small\hair 2pt
\endlabellist
\centering
\includegraphics[scale=1.5]{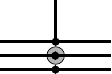}
\caption{At each $f$-vertex (highlighted), subsequent gluings 
attach at most one vertex of valence greater than two.  See also 
Figure~\ref{figure5}, in which the $f$-vertices are bold.}
\label{figure7}
\end{figure}

\subsection{Bounded folding}\label{subsection:bounded_folding}

For technical reasons, it is important to generalize this proposition and the definition of
$f$-foldedness to the case that $f:R \to R$ is not an immersion, but satisfies a slightly
weaker property, that we call {\em bounded folding}.

If $g:X \to Y$ is a map between graphs taking edges to edges, {\em Stallings folding} shows
how to construct canonically a quotient $\pi:X \to X'$ which is a map between graphs taking edges
to edges, and an immersion $X' \to Y$, so that the composition $X \to X' \to Y$ is $g$.

\begin{definition}
Let $g:X \to Y$ be a map of graphs, and let $X'$ be obtained by folding, so that 
$X'$ immerses in $Y$ and there is $\pi:X \to X'$ so that $X \to X' \to Y$ is $g$.
We say that $g$ has {\em bounded folding} if there is a collection of disjoint simplicial 
trees $T_i'$ in $X'$ so that each preimage $T_i:=\pi^{-1}(T_i')$ is a connected tree in $X$
containing at most one vertex of valence $>2$, 
and $\pi$ is a homeomorphism of $X-\cup_i T_i \to X' - \cup_i T_i'$ and
a proper homotopy equivalence of $T_i \to T_i'$ for each $i$.
Call the union of the $T_i$ the {\em folding region}, and denote it by $\fold(X)$; 
the complement of the folding region in $X$ is the {\em immersed region}.
\end{definition}

\begin{figure}[ht]
\labellist
\small\hair 2pt
 \pinlabel {$a$} at 50 30
 \pinlabel {$b$} at 63 35
 \pinlabel {$A$} at 76 40
 \pinlabel {$a$} at 37 31
 \pinlabel {$a$} at 24 36.5
 \pinlabel {$b$} at 7 43
 \pinlabel {$a$} at 36 11
 \pinlabel {$a$} at 24 5
 \pinlabel {$B$} at 9 -1
 \pinlabel {$b$} at 51 13
 \pinlabel {$a$} at 65 7
 \pinlabel {$B$} at 77 0
 
 \pinlabel {$b$} at 130 35
 \pinlabel {$a$} at 144 30
 \pinlabel {$b$} at 159 35
 \pinlabel {$a$} at 165.5 30
 \pinlabel {$b$} at 184 23
 \pinlabel {$a$} at 196 16
 \pinlabel {$B$} at 210 9
 \pinlabel {$B$} at 130 13
\endlabellist
\centering
\includegraphics[scale=1.5]{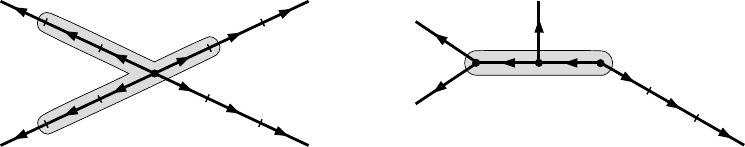}
\caption{The gray region, left, indicates all edges involved in folding 
(the \emph{folding region}). After folding, the gray region is reduced 
to the region at right.}
\label{figure8}
\end{figure}
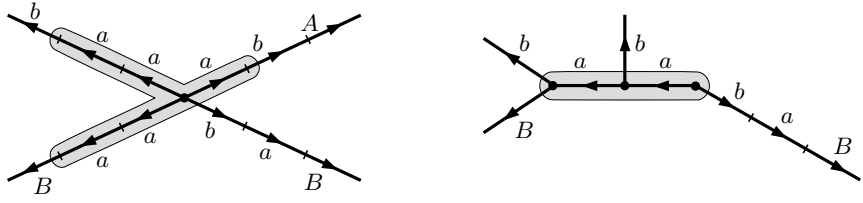

Note that $\fold(X)$ is precisely the preimage of the set of edges of $X'$ with
more than one preimage. Note also that if $g:X \to Y$ is a map with bounded folding, 
then $\pi:X \to X'$ is a homotopy equivalence, so $g$ is $\pi_1$-injective. 

Topologically, a map with bounded folding is an immersion outside a small tree neighborhood
of some vertices, and collapses each such neighborhood by a proper homotopy
equivalence to a smaller tree.

Now, the map $f:R \to R$ is not simplicial, since edges of $R$ get generally taken to long
paths in $R$. Let $R_1$ denote a rose with edges labeled by reduced words which are the image of the generators
of $F$ under $\phi:F \to F$ (assume none of these is trivial) and
subdivide edges of $R_1$ so that each edge gets one generator. Then we can factorize $f:R \to R$
as the composition of a {\em homeomorphism} $h_1:R \to R_1$ and a simplicial map $R_1 \to R$.

\begin{definition}
With notation as above, and by abuse of notation, we say that $f:R \to R$ 
has {\em bounded folding} if $R_1 \to R$ has bounded folding.
\end{definition}

If $f:R \to R$ has bounded folding, 
either $R_1 \to R$ is an immersion, or else $\fold(R_1)$ consists of a single tree 
with a single vertex of valence $>2$ which corresponds to the vertex of $R$ under
$h^{-1}$. See Figure~\ref{figure9}.

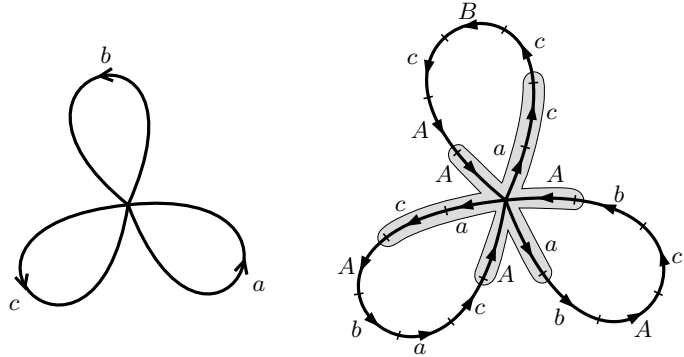
\begin{figure}[ht]
\labellist
\small\hair 2pt
 \pinlabel {$a$} at 61.5 14
 \pinlabel {$b$} at 23 72
 \pinlabel {$c$} at 0 9
 
 \pinlabel {$a$} at 135 24
 \pinlabel {$b$} at 137 6
 \pinlabel {$A$} at 158 3
 \pinlabel {$c$} at 166.5 21
 \pinlabel {$b$} at 152 37
 \pinlabel {$A$} at 136 42
 
 \pinlabel {$a$} at 122 48
 \pinlabel {$c$} at 135 57
 \pinlabel {$c$} at 132 75
 \pinlabel {$B$} at 114 83
 \pinlabel {$c$} at 100 71
 \pinlabel {$A$} at 102 53
 \pinlabel {$A$} at 108 42
 
 \pinlabel {$a$} at 112.5 28.5
 \pinlabel {$c$} at 97 34
 \pinlabel {$A$} at 83.5 19
 \pinlabel {$b$} at 86 3
 \pinlabel {$a$} at 102 -2
 \pinlabel {$c$} at 117 8
 \pinlabel {$A$} at 123.5 17
\endlabellist
\centering
\includegraphics[scale=1.5]{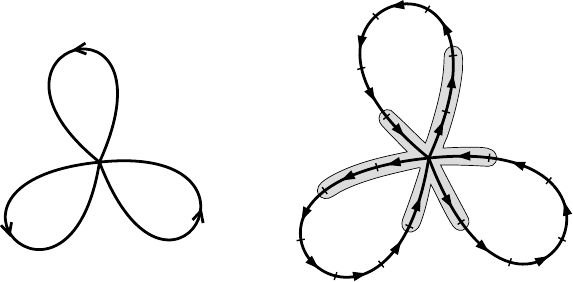}
\caption{The gray region indicates $\fold(R_1)$ for the endomorphism 
$a\mapsto abAcbA$, $b\mapsto accBcAA$, $c \mapsto acAbacA$.}
\label{figure9}
\end{figure}

Let $R_2$ be another rose whose edges are labeled by the {\em unreduced} words,
obtained by applying $\phi$ to each letter of the edge labels of $R_1$, and define $R_n$
similarly by induction. So there are homeomorphisms $h_n:R \to R_n$ and a simplicial map
$R_n \to R$ for which the composition $R \to R$ is $f^n$. By abuse of notation we also write
$f:R_{i-1} \to R_i$ for each $i$. Observe that $\fold(R_n)$ contains $f(\fold(R_{n-1}))$,
and the components of $\fold(R_n)-f(\fold(R_{n-1}))$ are intervals, none
of which contains the image of a vertex of $\fold(R_{n-1})$ (except possibly at an endpoint).

Now, suppose $g:X \to R$ is an $f$-folded $f$-fatgraph over $R$. 
We might be able to realize $\partial^- \to R$ by an immersion, but it is unlikely
that $\partial^+ \to R$ can be realized by an immersion if $f:R \to R$ is not an immersion.

\begin{definition}
Let $g:\partial^- \to R$ be an immersion, and let $h:\partial^+ \to R_1$ be obtained by
applying $f$ to both sides of $g$. Define $\Sigma^+$ to be the
preimage $\Sigma^+:=h^{-1}(\fold(R_1))$.
\end{definition}

Note that $\fold(\partial^+)$ is contained in $\Sigma^+$, which is a collection of 
intervals (it can't be all of $\partial^+$ because $\partial^- \to R$ is an immersion).

\begin{lemma}\label{lemma:peripheral_tree}
Let $w$ be a nonreduced cyclic word which is nontrivial, and let $V$ be the reduced cyclic word
which is inverse to $w$. Then $w \cup V = \partial S(Y(w))$ for an immersed fatgraph
$g:Y(w) \to R$ which consists of a circle (the embedded image of $V$) with a collection
of rooted trees attached, one for each component of $\fold(w)$.
\end{lemma}
\begin{proof}
This is just the observation that $w$ can be repeatedly Stallings 
folded to produce $v$ (the inverse of $V$); if we embed $w$ in the plane, the folds can
all be done to the ``inside'', producing a planar graph $Y(w)$ at the end with inner boundary
$V$ and outer boundary $w$. The embedding in the plane gives $Y(w)$ its fatgraph structure.
\end{proof}

If $w=\partial^+$ and $\Sigma^+$ is as above, each component of $\fold(w)$ is contained in
a component of $\Sigma^+$ and folds up to a tree in $Y(w)$ as
in Lemma~\ref{lemma:peripheral_tree}. The image of the component of $\Sigma^+$ is this
tree together possibly with an interval neighborhood of its root; we call this entire
image a {\em peripheral tree}, and denote the union of these trees by $\Sigma$. See Figure~\ref{figure10}.


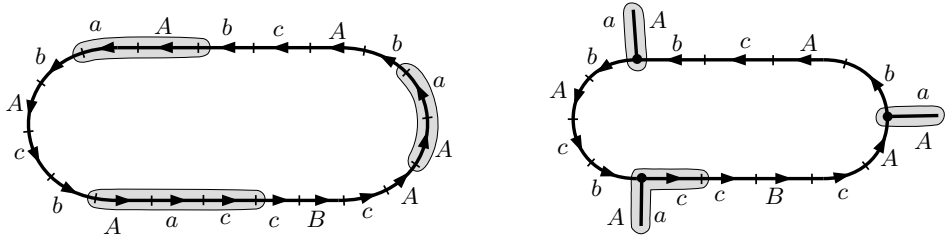
\begin{figure}[ht]
\labellist
\small\hair 2pt
 \pinlabel {$a$} at 98 36
 \pinlabel {$b$} at 88 45
 \pinlabel {$A$} at 74 48
 \pinlabel {$c$} at 60 48
 \pinlabel {$b$} at 48 49
 \pinlabel {$A$} at 33 49
 \pinlabel {$a$} at 17 49
 \pinlabel {$b$} at 4 42
 \pinlabel {$A$} at -2 31
 \pinlabel {$c$} at -1 19
 \pinlabel {$b$} at 8 7
 \pinlabel {$A$} at 21 2
 \pinlabel {$a$} at 35 3
 \pinlabel {$c$} at 47.5 3.5
 \pinlabel {$c$} at 59 3
 \pinlabel {$B$} at 69 3.5
 \pinlabel {$c$} at 81 5
 \pinlabel {$A$} at 91 9
 \pinlabel {$A$} at 99 20
 
 \pinlabel {$a$} at 212 33
 \pinlabel {$b$} at 204 38
 \pinlabel {$A$} at 185 45
 \pinlabel {$c$} at 170 45
 \pinlabel {$b$} at 154 45
 \pinlabel {$A$} at 149.5 50.5
 \pinlabel {$a$} at 138 50
 \pinlabel {$b$} at 134 42
 \pinlabel {$A$} at 126 34
 \pinlabel {$c$} at 127 20
 \pinlabel {$b$} at 135 11
 \pinlabel {$A$} at 139.5 4
 \pinlabel {$a$} at 150 4
 \pinlabel {$c$} at 155 8
 \pinlabel {$c$} at 164 9
 \pinlabel {$B$} at 177 9
 \pinlabel {$c$} at 193 9.5
 \pinlabel {$A$} at 204 18
 \pinlabel {$A$} at 212 22

\endlabellist
\centering
\includegraphics[scale=1.6]{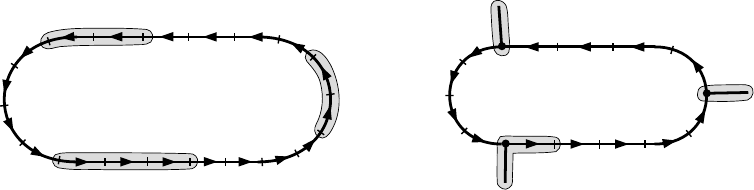}
\caption{Applying the endomorphism from Figure~\ref{figure9} to the loop $aab$ 
produces the loop at left, with the folding region in gray.  After folding, 
the folding region is reduced to a collection of peripheral trees, right.}
\label{figure10}
\end{figure}

\begin{definition}\label{definition:grafting}
Let $w$ be a possibly unreduced nontrivial cyclic word, 
and $\Sigma^+$ a collection of embedded intervals
containing $\fold(w)$. Let $Y(w)$ be as in the statement of
Lemma~\ref{lemma:peripheral_tree}, and let $\Sigma$ be the union of peripheral trees in 
$Y(w)$.

An inclusion of $Y(w)$ into another immersed fatgraph $X$ is a {\em grafting of $Y(w)$} 
if it satisfies the following properties:
\begin{enumerate}
\item{$w$ is a component of $\partial S(X)$;}
\item{all 1-valent vertices of $X$ are in $\Sigma$; and}
\item{every vertex of $\Sigma$ has the same valence in $Y(w)$ as in $X$.}
\end{enumerate}
Let $Y'$ be the fatgraph obtained from $X$ by cutting off the peripheral trees at their roots.
Then we say $X$ is obtained by {\em grafting $Y(w)$ onto $Y'$}.
\end{definition}

\begin{definition}
Suppose $f:R \to R$ has bounded folding, and let $X$ be an $f$-fatgraph $g:X \to R$
immersed in $R$. We say that $g:X \to R$ admits 
{\em bounded $f$-folding} if the following is true: 
\begin{enumerate}
\item{$X$ is obtained by grafting $Y(\partial^+)$, where as above 
$\Sigma^+ \subset \partial^+$ is defined to be $h^{-1}(\fold(R_1))$;}
\item{$g:X \to R$ is $f$-folded in the sense of Definition~\ref{definition:f_folded},
except that it is possible that some $f$-vertices in $\partial^+$ map to a 1-valent
vertex of $X$ on the boundary of a peripheral tree;}
\item{distinct $f$-vertices map to different components of $\Sigma$; and}
\item{the image of $\partial^-$ is disjoint from $\Sigma$.}
\end{enumerate}
\end{definition}

\begin{proposition}\label{proposition:bounded_f_folding_injective}
Suppose $f:R \to R$ has bounded folding, and $g:X \to R$ admits bounded $f$-folding. Then
$S*_f(X) \to K$ is $\pi_1$-injective.
\end{proposition}
\begin{proof}
We can build a surface $S_n(X)$ and a fatgraph $X_n$ as before, where $S_n(X)=S(X_n)$, 
since $\partial^- \to X$ is an embedding, and Lemma~\ref{lemma:embedding}. 

We claim that $X_n \to R$ has bounded folding, 
and is therefore $\pi_1$-injective. 
We build $X_n$ from $X$ and $f(X_{n-1})$, by
gluing $\partial^+$ in $X$ to $f(\partial^-)$ in $f(X_{n-1})$. Note that 
the inclusion of $X$ in $X_n$ is an {\em embedding}, since $X$ is attached
by identifying $\partial^+$ with $f(\partial^-)$ which {\em embeds} in $f(X_{n-1})$. 

We assume by induction that $X_{n-1}$ has bounded folding. Then so does $f(X_{n-1})$, since 
$\fold(f(X_{n-1}))-f(\fold(X_{n-1}))$ consists of a union of small intervals, none of which
contains the image of a vertex of $X_{n-1}$ except possibly at the endpoints (this is a general
property of the fact that $f$ has bounded folding and $g$ is an immersion).

We need to check that no two vertices of $X_n$ of valence at least 3 are contained in
the same component of $\fold(X_n)$. The vertices of $X_n$ of valence at least 3 are all
images of a vertex of valence at least 3 either in $f(X_{n-1})$ or in $X$. Moreover,
the vertices of valence at least 3 in $f(X_{n-1})$ are the images of vertices of
valence at least 3 in $X_{n-1}$. By abuse of notation, we
refer to the images of {\em all} the vertices of $X_{n-1}$ in $f(X_{n-1})$ as
$f$-vertices; the ordinary $f$-vertices in $X$ are glued to the $f$-vertices (in the new sense)
of $f(\partial^-)$. 

Every $f$-vertex in $f(X_{n-1})$ not in $f(\partial^-)$ is thus separated
from the image of $X$ in $X_n$ by the image of an edge of $X_{n-1}-\partial^-$, and
the endpoints of this edge are necessarily in different components of $\fold(X_n)$. Distinct
$f$-vertices in $f(\partial^-)$ must map to distinct vertices of $X$, and no component of
$\Sigma$ contains the image of more than one of them, by condition (3); thus
components of $\fold(X_n)$ cannot contain more than one such $f$-vertex.

So we just need to check that distinct high-valence vertices of $X$ are not included
into the same component of $\fold(X_n)$. Now, it is not necessarily true that
$\fold(X_n) \cap X$ is equal to $\fold(X)$, but the difference is contained in
$\Sigma$ minus the peripheral trees (i.e.\/ in the intervals of $X$ containing the
roots of the peripheral trees) and by the defining properties of grafting, there are no other
high valence vertices there.
\end{proof}

\begin{remark}
If one is prepared to work with {\em groupoid} generators for $F$ rather than group generators,
this contents of this section are superfluous in most cases of interest. Although most 
injective endomorphisms $\phi:F \to F$ are not represented by immersions of some rose $R$,
Reynolds \cite{Reynolds} showed that if $\phi$ is an {\em irreducible endomorphism} which is
{\em not} an automorphism, then there is some graph $R'$ (typically with more than one vertex)
and an isomorphism of $F$ with $\pi_1(R')$, so that $\phi$ is represented by an immersion
$f:R' \to R'$. If one wants to find injective surface subgroups in extensions $F*_\phi$ then 
in practice it is much easier to work with $f$-folded $f$-fatgraphs over such an $R'$, than
with boundedly $f$-folded $f$-fatgraphs over a rose $R$.
\end{remark}

\subsection{Random endomorphisms}

\begin{definition}
Fix a free group $F$ and a free generating set. A {\em random endomorphism of length $n$}
is an endomorphism $\phi:F \to F$ which takes each generator to a reduced word of length $n$
sampled randomly and independently from the set of all reduced words of length $n$ with
the uniform distribution.
\end{definition}

We require an elementary lemma from probability:

\begin{lemma}\label{lemma:small_prefix}
Let $\phi:F \to F$ be a random endomorphism of length $n$. There for any positive constant $C$
there is a positive constant $c$
depending only on the rank of $F$ so that with probability $1-O(e^{-n^c})$,
for every two distinct generators or inverses of generators $x$, $y$ the reduced words
$\phi(x)$ and $\phi(y)$ have a common prefix or suffix of length $\le C\log{n}$.
\end{lemma}
\begin{proof}
It suffices to obtain such an estimate for two random words. Generate the
words letter by letter; at each step the chance that there is a mismatch is at least
$(k-1)/k$. The estimate follows.
\end{proof}

It follows that if $\phi:F \to F$ is a random endomorphism, the map $f:R \to R$
has bounded folding, and the diameter of $\fold(R_1)$ in $R_1$ is at most $2C\log{n}$,
with probability $1-O(e^{-n^c})$, where we may choose $C$ as small as we like at the cost of
making $c$ small.

We now come to the main theorem of this section, the Random $f$-folded Surface Theorem:

\begin{theorem}[Random $f$-folded surface]\label{thm:random_f_folded_theorem}
Let $k\ge 2$ be fixed, and let $F$ be a free group of rank $k$. Let $\phi$ be a
random endomorphism of $F$ of length $n$. Then the probability that $F*_\phi$ contains
an essential surface subgroup is at least $1-O(e^{-n^c})$ for some $c>0$.
\end{theorem}

We will prove this theorem by constructing an $f$-fatgraph $X$ for which $g:X \to R$
admits bounded $f$-folding, and then apply 
Proposition~\ref{proposition:bounded_f_folding_injective}.

In the sequel we denote generators by smaller case letters $a,b,c$ and so on, 
and their inverses by capitals; thus $A:=a^{-1}$, $B:=b^{-1}$ etc. 
Let $a,b$ be two generators of $F$, let $\partial^-$ be an oriented circle labeled with the
(reduced) cyclic word $abAB$ and let $\partial^+$ be an oriented circle labeled with the
(possibly unreduced!) word obtained by cyclically concatenating 
$\phi(b)$, $\phi(a)$, $\phi(B)$, $\phi(A)$. Note that the label on $\partial^+$ is equal
to the inverse of $\phi(abAB)$ in $F$. We will construct $X$ with $\partial S(X) = \partial^-
\cup \partial^+$ with notation as in Definition~\ref{definition:f_fatgraph}.

We build $X$ as a graph by starting with $\partial^+ \cup \partial^-$ and identifying pairs
of segments with opposite orientations and inverse labels. At each stage,
we obtain a {\em partial fatgraph} (bounding the pairs of edges that have been identified)
and a {\em remainder}. See Figure~\ref{figure11}.

\begin{figure}[ht]
\labellist
\small\hair 2pt
 \pinlabel {$a$} at 132 74
 \pinlabel {$b$} at 111 82
 \pinlabel {$b$} at 89 82
 \pinlabel {$b$} at 65 82
 \pinlabel {$a$} at 44 77
 \pinlabel {$a$} at 42 54
 \pinlabel {$B$} at 62 45
 \pinlabel {$A$} at 85 46
 \pinlabel {$B$} at 107 46
 \pinlabel {$a$} at 129 51
 
 \pinlabel {$A$} at 179 63
 \pinlabel {$A$} at 160 82
 \pinlabel {$B$} at 139 64
 \pinlabel {$A$} at 158 45

 \pinlabel {$a$} at 89 39                                                                                                                                  
 \pinlabel {$b$} at 59 28                                                                                                                                  
 \pinlabel {$b$} at 40 26                                                                                                                                  
 \pinlabel {$b$} at 16 30                                                                                                                                  
 \pinlabel {$a$} at -2 20 
 \pinlabel {$a$} at 15 8                                                                                                                                   
 \pinlabel {$B$} at 37 13                                                                                                                                  
 \pinlabel {$A$} at 59 9.5                                                                                                                                   
 \pinlabel {$B$} at 86 -1.5
 \pinlabel {$a$} at 100 18                                                                                                                                 
 
 \pinlabel {$B$} at 84 24.5                                                                                                                                  
 \pinlabel {$A$} at 79.4 19.5
 \pinlabel {$A$} at 84 14.4                                                                                                                                  
 \pinlabel {$A$} at 89 19                                                                                                                                  
 
 \pinlabel {$a$} at 196 41                                                                                                                                 
 \pinlabel {$b$} at 171 29                                                                                                                                 
 \pinlabel {$b$} at 155 26                                                                                                                                 
 \pinlabel {$b$} at 141 26                                                                                                                                 
 \pinlabel {$b$} at 124 30                                                                                                                                 
 \pinlabel {$a$} at 106 25                                                                                                                                 
 \pinlabel {$a$} at 123.5 8.5                                                                                                                                  
 \pinlabel {$B$} at 140 13.5                                                                                                                                 
 \pinlabel {$B$} at 153 13
 \pinlabel {$A$} at 171 11
 \pinlabel {$B$} at 191 -2
 \pinlabel {$a$} at 207 -2
 \pinlabel {$A$} at 212 8
 \pinlabel {$A$} at 207 14
 \pinlabel {$A$} at 203.5 20
 \pinlabel {$B$} at 206 27
 \pinlabel {$A$} at 213 34
\endlabellist
\centering
\includegraphics[scale=1.5]{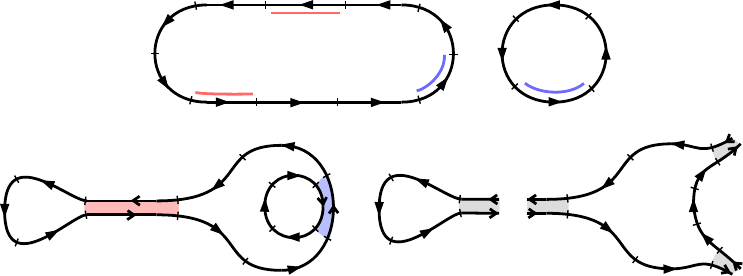}
\caption{To build a fatgraph with desired boundary loops, we can proceed by 
gluing small portions of the loops, one at a time.  After gluing a small amount of our 
loops, we obtain a partial fatgraph and the remainder, which is a 
collection of loops with tags.}
\label{figure11}
\end{figure}
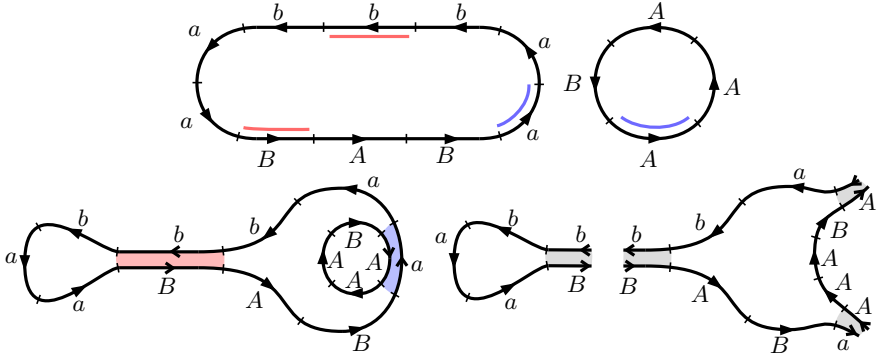

When all of $\partial^+ \cup \partial^-$
has been paired (i.e.\/ when the remainder is empty), 
the result will be $X$. The proof will take up the next section.

\section{Proof of the Random $f$-folded Surface Theorem}

\subsection{Bounded folding and $f$-vertices}

The loop $\partial^-$ has length 4 and thus 4 vertices. The loop $\partial^+$ has length $4n$,
and has 4 $f$-vertices, which separate it into the segments $\phi(b)$, $\phi(a)$, $\phi(B)$
and $\phi(A)$. The map $\partial^- \to R$ is an immersion already. Also, by
Lemma~\ref{lemma:small_prefix} we have already seen that $\fold(R_1)$ is a tree in $R_1$
containing the vertex, of diameter at most $2C\log{n}$, where $C$ is as small as we like,
with probability $1-O(e^{-n^c})$. Since $h:\partial^+ \to R_1$ (obtained by applying $f$
to both sides of $\partial^- \to R$) is an immersion, it follows that $\Sigma^+:=h^{-1}(\fold(R_1))$
consists of four neighborhoods of the $f$-vertices, each of diameter at most $2C\log{n}$.

Note that $\fold(\partial^+)$ is contained in $\Sigma^+$. We fold $\Sigma$ as much as 
possible, obtaining a fatgraph $Y(\partial^+)$ as in Lemma~\ref{lemma:peripheral_tree}
with $\partial^+$ on one side of $S(Y(\partial^+))$ and with (the inverse of)
the reduced representative of this word on the other side. Denote the image of $\Sigma^+$
in this fatgraph by $\Sigma$, and let $\Sigma^-$ be the {\em reduced} words on the other
side of $S(Y(\partial^+))$ (i.e.\/ they are the reduced words obtained from the components
of $\Sigma^+$). Note that $\Sigma^-$ has at most four components, each of length
at most $2C\log{n}$. Notice also that if a component of $\Sigma^+$ can be reduced at all,
it can only be reduced by cancelling a pair of maximal inverse subwords on the sides of
the $f$-vertex, so that the peripheral trees consist of at most a single interval with
the $f$-vertex at the tip.

\begin{lemma}\label{lemma:log_word_present}
For any positive $\epsilon$, if $w$ is a random reduced word in $F$
of length at least $n\epsilon$, then for any $C' < 1/\log(2k-1)$ there is a positive $c$
(depending only on the rank of $F$), so that for any reduced word $v$ of length
$\lfloor C'\log{n} \rfloor$ we can find a copy of $v$ in $w$, with probability $1-O(e^{-n^c})$.
\end{lemma}
For a proof, see e.g.\/ \cite{CW_rigidity}, Prop.~2.3 which gives a precise count of
the number of copies of $v$ in $w$. So if we choose $C \ll 1/\log(2k-1)$ we can find
many disjoint copies of segments in $\partial^+-\Sigma^+$ with labels inverse 
to the labels on $\Sigma^-$, and we can pair these segments. 

Next, we look for a copy of the word $bbaBAB$ in the remainder, glue the outermost
copies of $b$ and $B$, and glue the resulting $baBA$ loop to $\partial^-$. Note that
$\partial^-$ embeds into the resulting partial fatgraph, and is disjoint from 
$\Sigma$ and the $f$-vertices.

The part of the fatgraph we have built so far evidently immerses in $R$.
All that is left of the remainder are reduced cyclic words made from the
segments of $\partial^+$ which are disjoint from
$\Sigma$ and the $f$-vertices. It remains to glue up the remainder so that
the resulting fatgraph is immersed. This is a complicated combinatorial argument with
several steps, and it takes up the remainder of the section.

\subsection{Pseudorandom words}

At this stage of the construction, the remainder consists of a collection of cyclic words made from
the 9 segments in $\partial^+$ disjoint from $\Sigma$ and the $f$-vertices. We can arrange for each
of these segments to be long (i.e.\/ $O(n)$), so that in effect we can think of the remainder
as a finite collection of long reduced cyclic words in $F$ whose sum is homologically trivial.

Note that although each segment making up the cyclic words is (more or less) random, different segments
are not necessarily independent --- some of them are subwords of $\phi(a)$ and
some are subwords of $\phi(A)$, which are inverse. But each individual segment is {\em pseudorandom}
in the following sense.

\begin{definition}
For $T>0$ and $\epsilon>0$, a reduced word $w$ or cyclic word in a free group $F$ of rank $k$
is {\em $(T,\epsilon)$-pseudorandom} if, however
we partition $w$ as
$$w=w' v_1 v_2 \cdots v_\ell w''$$
where $|w'|$ and $|w''|<T$ and where $|v_i|=T$ for each $i$, and for any reduced word
$\sigma$ in $F$ of length $T$, there is an inequality
$$1-\epsilon \le \frac {\text{number of $v_i$ equal to $\sigma$}} {\ell/2k(2k-1)^{T-1}} \le 1+\epsilon$$
\end{definition}
Here the term $2k(2k-1)^{T-1}$ is simply the number of reduced words in $F$ of length $T$, so this
just says that the subwords of $w$ of length $T$ are distributed uniformly, up to multiplicative
error $\epsilon$.

Now, for any fixed $T,\epsilon$, a random reduced word in $F$ of length $n$ will be $(T,\epsilon)$-pseudorandom
for sufficiently big $n$, with probability $1-O(e^{Cn})$ for some $C>0$. This follows from
\cite{CW_rigidity}, Prop.~2.3; in fact, with probability $1-O(e^{n^c})$, one can even let $T$
grow with $n$, at the rate $T=C'\log{n}$ for suitable $C'$ (compare with Lemma~\ref{lemma:log_word_present}).

It follows that for big $n$, each individual subword $\phi(a)$, $\phi(b)$ and their inverses is 
$(T,\epsilon)$-pseudorandom with high probability, and so are all their subwords of length $\delta n$
for any fixed positive $\delta$. Thus, after the first stage of the construction, the
remainder consists of a finite collection of cyclic loops, each of which is $(T,\epsilon)$-pseudorandom for any fixed
$T,\epsilon$. Theorem~\ref{thm:random_f_folded_theorem} will therefore be proved if we can show that
any finite collection of $T,\epsilon$-pseudorandom reduced cyclic words whose sum is homologically trivial bounds
a folded fatgraph.

\subsection{Folding off short loops}

We introduce some notation to simplify the discussion in what follows.

In order to distinguish words (with a definite initial letter) from cyclic
words, we delimit a word (in our notation) by adding centered dots on both sides; thus $\cdot w \cdot$
is a word, and $w$ is the corresponding cyclic word. 

Furthermore, in the course of our folding, we will obtain segments which are in the boundary of a
partial fatgraph, and it is important to indicate which vertices have valence bigger than 2 in the
fatgraph. We will insist that all vertices of valence $>2$ in the remainder at each stage
will be 3-valent, and use the notation $\perp$ for such a vertex. If it is important to record the
label on the third edge at this vertex, we denote it $\perp^x$ where $x$ is the outgoing letter.
Thus, $\cdot ab\perp^ba\cdot$ denotes a segment in the remainder of length 3 with the label $aba$,
and after the second letter there is a 3-valent vertex in the partial fatgraph, with outgoing edge
label $b$.

So the remainder at this stage consists of a finite collection of cyclic words of the form
$\cdots u_i \perp^{x_i} u_j \perp^{x_j} u_k \cdots$ where each $u_i$ is one of the 9 segments
of $\partial^+$, and $x_i$ is the outgoing letter on the edge of the partial fatgraph built by
the identifications made so far.

Now let $w$ be a $(T,\epsilon)$-pseudorandom word. 
We perform the following process. As we read off the
letters of $w$ one by one, we look for a segment $\sigma$ of length 11 of the form $\cdot v_1Pupv_2\cdot$ 
satisfying the following properties:
\begin{enumerate}
\item{$|v_1|=|v_2|=1$ and $v_1\ne v_2^{-1}$;}
\item{$|p|=|P|=1$ and $P=p^{-1}$; and}
\item{$|u|=7$ and $u$ is {\em cyclically} reduced.}
\end{enumerate}
Then $p$ and $P$ can be glued, producing a new reduced word $w'$ containing $\cdot v_1\perp^P v_2\cdot$ 
where $\sigma$ was, and a {\em short loop} with the cyclic word $u\perp^p$ on it. We call this
operation {\em folding off a short loop}. The data of short loop is determined by a word $u$ of length
$7$ whose associated cyclic word is cyclically reduced, together with a letter $p$ not equal to
the first letter of $u$ or the inverse of the last letter. This data $(u,p)$ is called the {\em type}
of a short loop. Let $L_k$ denote the number of distinct types of short loops in a free group, so
for example $L_2=4376$.

As we read through a component of the remainder, we fold off short loops at regular intervals whenever
we can, so that the ``stems'' of the loops land at places separated by intervals of even length
(say).  To fold off a short loop, we desire the pattern described above, and the segments which 
satisfy the pattern are simply a subset of all segments of length $11$.  
Because $w$ is $(T,\epsilon)$-pseudorandom and there are finitely many types of segments, 
whenever $T$ is large enough, we will find short loops of all kinds, 
and they will be nearly equidistributed, as described more fully below.

Thus we obtain in this way a {\em reservoir} of short loops, together with the rest of the
remainder, which is a collection of long cyclic words with many trivalent vertices at the
steps of the short loops, where adjacent trivalent vertices are separated by intervals of even
length (with the possible exception of the nine trivalent vertices associated to the vertices
of the fatgraph produced at the first step). See Figure~\ref{figure12}.

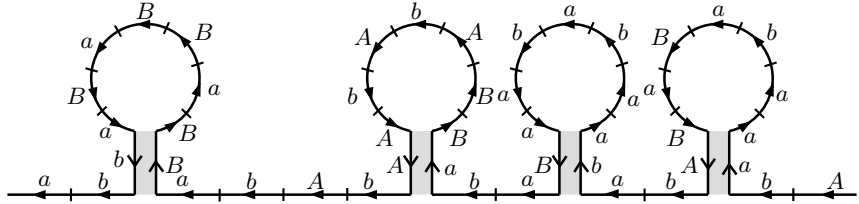
\begin{figure}[ht]
\labellist
\small\hair 2pt
 \pinlabel {$A$} at 156 4
 \pinlabel {$b$} at 143 4
 
 \pinlabel {$a$} at 139 6.5
 \pinlabel {$a$} at 140 12
 \pinlabel {$a$} at 146 20
 \pinlabel {$b$} at 144 32
 \pinlabel {$a$} at 134 36
 \pinlabel {$B$} at 123 31
 \pinlabel {$a$} at 121 21
 \pinlabel {$B$} at 126 12
 \pinlabel {$A$} at 128.5 7
 
 \pinlabel {$b$} at 125 4
 \pinlabel {$a$} at 115 4
 
 \pinlabel {$b$} at 111 7                                                                                                                             
 \pinlabel {$a$} at 112 12                                                                                                                            
 \pinlabel {$a$} at 118 19                                                                                                                            
 \pinlabel {$b$} at 116 32
 \pinlabel {$a$} at 106 36                                                                                                                          
 \pinlabel {$b$} at 96 32                                                                                                                             
 \pinlabel {$a$} at 93.5 20                                                                                                                             
 \pinlabel {$a$} at 98 13                                                                                                                             
 \pinlabel {$B$} at 101 7                                                                                                                              
 
 \pinlabel {$a$} at 98 4                                                                                                                              
 \pinlabel {$b$} at 88 4                                                                                                                             
 
 \pinlabel {$a$} at 83.5 6                                                                                                                              
 \pinlabel {$B$} at 85 12                                                                                                                             
 \pinlabel {$B$} at 90 20                                                                                                                             
 \pinlabel {$A$} at 88 32                                                                                                                             
 \pinlabel {$b$} at 77 36                                                                                                                             
 \pinlabel {$A$} at 67 31                                                                                                                             
 \pinlabel {$b$} at 65 20                                                                                                                             
 \pinlabel {$A$} at 71 12                                                                                                                              
 \pinlabel {$A$} at 73 7
                                                                                                                               
 \pinlabel {$b$} at 68 4                                                                                                                              
 \pinlabel {$A$} at 58 4                                                                                                                              
 \pinlabel {$b$} at 46 4                                                                                                                              
 \pinlabel {$a$} at 33 4
                                                                                                                               
 \pinlabel {$B$} at 31.5 7                                                                                                                              
 \pinlabel {$B$} at 34 13                                                                                                                             
 \pinlabel {$a$} at 39 21                                                                                                                             
 \pinlabel {$B$} at 37 32                                                                                                                             
 \pinlabel {$B$} at 26 36
 \pinlabel {$a$} at 15 30
 \pinlabel {$B$} at 13 20
 \pinlabel {$a$} at 18.5 13
 \pinlabel {$b$} at 21 8
 
 \pinlabel {$b$} at 18 4
 \pinlabel {$a$} at 7 4

\endlabellist
\centering
\includegraphics[scale=2.0]{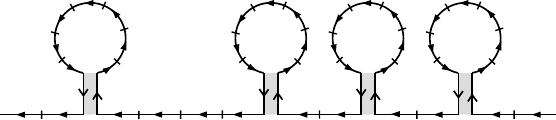}
\caption{Folding off short loops.  A loop can be folded off only when 
all of the three properties are satisfied.  This ensures the 
result is a collection of tagged loops of length exactly $7$.}
\label{figure12}
\end{figure}

Let $T$ be some big odd number. We perform this folding procedure on each successive
subword $v_i$ of a $(T,\epsilon)$-pseudorandom $w$ of length $T$ and obtain a collection of new words
$v_i'$ with trivalent vertices separated by even length intervals, 
such that length of $v_i'$ and that of $v_i$ agree mod 9. Since $T$ is odd, we distinguish between
{\em even} $v_i'$, for which the trivalent vertices are an even distance from the initial vertex, and
{\em odd} $v_i'$, for which the trivalent vertices are an odd distance from the initial vertex.

By pseudorandomness, the reservoir
consists of an {\em almost equidistributed} collection of short loops; i.e.\/ the distribution
differs from the uniform distribution by a multiplicative error of $\epsilon$. Furthermore, the
words $v_i'$ themselves are uniformly distributed with multiplicative error $\epsilon$, for each
fixed possible value of $|v_i'|$, and the same is true if one conditions on
the $v_i'$ being even or odd (in the sense above). This is because the distribution of random letters $v_2$ not equal
to $v_1^{-1}$ that follow a subword of the form $v_1Pup$ averaged over all possible $Pup$ is
just the uniform distribution on letters not equal to $v_1^{-1}$. 

\subsection{Random pairing of $v_i'$}

The $v_i'$ fall into finitely many families depending on their lengths and {\em parity} --- i.e.\/ whether
they are even or odd in the sense of the previous subsection. Moreover, for each fixed
length and parity, the distribution of words is uniform up to multiplicative error $\epsilon$. Thus, for every
reduced word $\sigma$ of suitable length, 
the number of $v_i'$ of any given parity equal to $\sigma$ and the number equal to $\sigma^{-1}$ is very nearly equal.

For each $v_i'$ let $v_i''$ denote the subword of $v_i'$ excluding the first and last letter. We call a pair
$v_i'$ and $v_j'$ {\em compatible} if they satisfy the following conditions:
\begin{enumerate}
\item{the label on $v_i''$ is $\sigma$, and the label on $v_j''$ is $\sigma^{-1}$, for some $\sigma$ reduced;}
\item{the first letter of $v_i'$ is not inverse to the last letter of $v_j'$, and conversely;}
\item{if $|\sigma|$ is even, the parities of $v_i'$ and $v_j'$ are opposite, and if $|\sigma|$ is odd,
the parities agree.}
\end{enumerate}

Let $v_i'$ and $v_j'$ be compatible. Then we can glue $v_i''$ to $v_j''$, and by condition (3)
the trivalent vertices on either side are not identified. Furthermore, because of condition (2) the
new boundary words that result from the gluing are still reduced. See Figure~\ref{figure13}.

\begin{figure}[ht]
\labellist
\small\hair 2pt
 \pinlabel {$a$} at 158.5 19
 \pinlabel {$b$} at 151 15.5
 \pinlabel {$b$} at 147 17
 \pinlabel {$B$} at 136 18
 \pinlabel {$a$} at 135 14.8
 \pinlabel {$a$} at 123 14.8
 \pinlabel {$B$} at 119.7 18
 \pinlabel {$b$} at 108.5 19
 \pinlabel {$b$} at 104.9 15
 \pinlabel {$A$} at 95 15.5
 \pinlabel {$B$} at 91.5 18
 \pinlabel {$b$} at 81 19
 \pinlabel {$b$} at 78.5 15.2
 \pinlabel {$a$} at 68 15.5
 \pinlabel {$B$} at 57 15.6
 \pinlabel {$B$} at 43 15.6
 \pinlabel {$a$} at 39.3 18
 \pinlabel {$A$} at 28.2 19
 \pinlabel {$b$} at 26.3 15.5
 \pinlabel {$a$} at 14 15.5
 \pinlabel {$b$} at 6 19
 
 \pinlabel {$b$} at 6 2
 \pinlabel {$A$} at 11 5
 \pinlabel {$A$} at 15 2
 \pinlabel {$a$} at 25.4 1
 \pinlabel {$B$} at 28 5
 \pinlabel {$b$} at 40 4.8
 \pinlabel {$a$} at 44 2
 \pinlabel {$A$} at 53.5 2
 \pinlabel {$b$} at 56 5
 \pinlabel {$A$} at 65.5 5
 \pinlabel {$A$} at 69 2
 \pinlabel {$a$} at 79 1
 \pinlabel {$B$} at 81 5
 \pinlabel {$a$} at 92 5
 \pinlabel {$B$} at 106 5
 \pinlabel {$A$} at 119 5
 \pinlabel {$b$} at 124 2
 \pinlabel {$B$} at 134 2
 \pinlabel {$A$} at 137 5.5
 \pinlabel {$B$} at 149 5
 \pinlabel {$a$} at 158 2
\endlabellist
\centering
\includegraphics[scale=2.1]{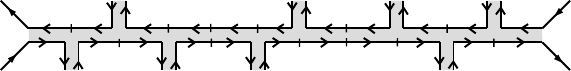}
\caption{Random pairing of the $v_i'$.  The shorter region 
which is entirely glued is $v_i''$.  Note that the initial and final letters
of the paired words are \emph{not} inverse, so the gluing is limited to 
exactly the $v_i''$.}
\label{figure13}
\end{figure}
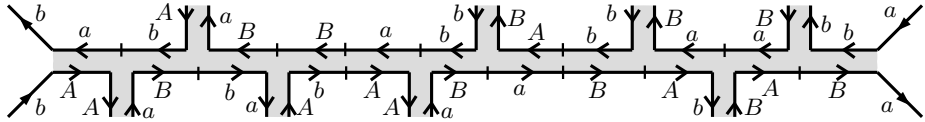

By pseudorandomness, we can glue all but $\epsilon$ of the total length of the remainder (excluding
the reservoir) in this way, and we are left with some collection $\Gamma$ of cyclic words, where
$|\Gamma|\le \epsilon n$, plus the reservoir.

\subsection{Gluing up $\Gamma$}

The next step of the construction is elementary. We use some relatively small mass of small loops from the
reservoir to pair up with $\Gamma$, so at the end of this step we are left only with loops in the
reservoir. Furthermore, since (for a suitable choice of $\epsilon$) the mass of $\Gamma$ is so small, 
even compared to the mass of short loops of any given type, if we can do this construction
while using at most $|\Gamma|$ short loops, the content of the reservoir at the end will still be
almost equidistributed, with multiplicative error some new (but arbitrarily small) constant
$\epsilon'$.

We claim that for any positive $m$ we can glue $m-6$ (or at most $3$ if $m<7$) short loops together 
to create a trivalent partial fatgraph with unglued part a loop of length $m$. The cases $m<7$
are elementary, and for $m=7$ one can take a single short loop with no gluing at all. We prove the
general case by induction, by proving the stronger statement that the trivalent fatgraph of length
$m\ge 7$ can be chosen to contain a pair of adjacent segments in its boundary of length $4$ and $3$ each
containing no trivalent vertex in the interior. This is obviously true for $m=7$. Suppose it is true for $m$,
and denote the adjacent segments by $\cdot 4\perp 3\cdot$. We can ``bracket'' this as 
$\cdot 3 (1\perp 2)2\cdot$ and bracket another short loop as $4(1\perp 2)$. Gluing the two (bracketed)
segments of length 3, we see obtain a loop of length $m+1$ containing $\cdot 3\perp 4 \perp 2\cdot$,
completing the induction step and proving the claim. See Figure~\ref{figure14}.

\begin{figure}[ht]
\labellist
\small\hair 2pt 

\endlabellist
\centering
\includegraphics[scale=1.5]{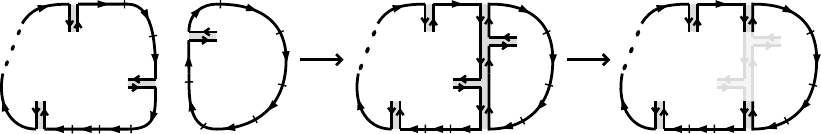}
\caption{Attaching a loop of length 7 as shown increases the total 
loop length by $1$.  By induction, we create a trivalent partial fatgraph 
with one unglued loop of any size.  We can then label the fatgraph 
arbitrarily such that the unglued loop is any desired word.}
\label{figure14}
\end{figure}

If we choose suitable short loop types,
the labels on the resulting partial fatgraph and the edges incident to the trivalent vertices
can be {\em arbitrary}, so we can build a loop that can be used to cancel a loop of $\Gamma$.

\subsection{Linear programming}

Finally, we are left with a reservoir of almost equidistributed short loops. Since our original
collection of words had homologically trivial sum, the same is true for the reservoir. We will show
that any homologically trivial collection of almost equidistributed short loops can be glued up
entirely, thus completing the construction of the $f$-folded fatgraph $X$, and the proof of
Theorem~\ref{thm:random_f_folded_theorem}. In fact, it is easier to show that a {\em multiple}
of any such collection can be glued up; thus the fatgraph we build will be assembled from the
disjoint union of copies of the partial fatgraphs built so far, glued up along the collection of
short loops, and $\partial^\pm$ will consist of the same number of disjoint copies of $[a,b]$ and $\phi([b,a])$.
This is evidently enough to prove the theorem. 

The advantage of allowing ourselves to use multiple copies is that we can find a solution to
the gluing problem ``over the rationals''. Formally, let $\L$ denote the vector space spanned by
the set of types of short loops, let $\L^+$ denote the cone of vectors with non-negative coefficients,
and let $\L^+_0$ denote the subcone of homologically trivial vectors. The ``uniform'' vector $\1$ is
the vector with all coefficients 1. This is in $\L^+_0$. Denote by $C$ the non-negative linear span
of the vectors representing collections that can be glued up. Then any rational vector in $C$ can be
``projectively'' glued up. It is easy to see that $\1$ is in $C$; we will show that $C$ contains an
open neighborhood of the ray spanned by $\1$. Since our collection of short loops is almost
equidistributed, it will be contained in this open neighborhood, and we will be done.
We call a vector {\em feasible} if it is in $C$.  

\begin{proposition}\label{prop:rank_2_trivalent}
In the above notation and in a rank $2$ free group, 
$C$ contains an open projective neighborhood of the ray spanned by $\1$.  
That is, there is some $\epsilon>0$ such that $C$ contains an 
$\epsilon$ neighborhood of $\1$ and thus contains all scalar multiples of 
this neighborhood.
\end{proposition}

Proposition~\ref{prop:rank_2_trivalent} is actually proved with a computation.  
We defer the proof and show how the arbitrary rank case reduces to rank 2.  
First we give some notation necessary for the proof.
We define a function 
$\iota$ on tagged loops which takes a loop $v$ to $v^{-1}$, with the tag in the 
diametrically opposite position.  When the tag switches positions under $\iota$, 
there is a choice about what the tag becomes, because there is more than one possible tag 
at each location in a word.  
Choose tags arbitrarily such that $\iota$ is an involution.  Notice that 
for any loop $\gamma$, there is an annulus (trivalent fatgraph) with boundary 
$\gamma + \iota(\gamma)$.  We call this an \emph{$\iota$-annulus}.

\begin{proposition}\label{prop:higher_rank_trivalent}
In the above notation and for any finite rank free group, 
$C$ contains an open projective neighborhood of the ray spanned by $\1$.
\end{proposition}
\begin{proof}
Suppose we are given any vector $v \in \L'^+_0$.  We must show that there is some 
$n$ such that $v + n\1$ bounds a trivalent fatgraph.  This will prove 
the proposition.  Therefore, we need to understand when a collection 
of loops, plus an arbitrary multiple of $\1$, bounds a trivalent fatgraph.

Given a collection of tagged loops $S$ and a single loop $\gamma$, 
suppose we can exhibit a trivalent fatgraph $Y$ with boundary 
$\gamma + \sum_i \alpha_i$.  We say that $\gamma$ is 
\emph{fatgraph equivalent} to $\sum_i\iota(\alpha_i)$.  Now, if 
we can find a trivalent fatgraph with boundary $S + \sum_i \iota(\alpha_i)$, 
then the union of this trivalent fatgraph with $Y$ has boundary 
$S + \sum_i\iota(\alpha_i) + \alpha_i + \gamma$.  In other words, 
we have a trivalent fatgraph with boundary $S + \gamma$, plus two $\iota$ pairs.  
If we then add all other remaining $\iota$ pairs, we can simply add 
$\iota$-annuli to get a trivalent fatgraph with boundary $S + \gamma + \1$.

Therefore, for the purpose of finding a trivalent fatgraph with a given boundary $S$ 
modulo adding some multiple of $\1$, if we have some loop $\gamma \in S$, 
and we find a collection of loops $\sum_i \alpha_i$ which is fatgraph equivalent to $\gamma$, 
then we can throw out $\gamma$ from $S$, replace it with $\sum_i \alpha_i$, and find 
a trivalent fatgraph bounding what remains.
To reduce to rank $2$, we'll show that any loop is fatgraph equivalent to a collection 
of loops, each of which contains only two generators.  After showing this, 
we'll explain how to apply Proposition~\ref{prop:rank_2_trivalent} to complete 
the proof.

\begin{lemma}\label{lem:equiv_to_rank_2}
Every loop is fatgraph equivalent to a collection of loops, 
each containing at most two generators.
\end{lemma}
\begin{proof}
To show that a loop is fatgraph equivalent to another collection of 
loops, note that it suffices to show it for \emph{un}tagged loops, provided 
there are positions on the fatgraph $Y$ to place tags.  This is because 
a loop is fatgraph equivalent to itself with a different tag position, for 
$\gamma + \gamma'^{-1}$ bounds a trivalent tagged annulus, where $\gamma'^{-1}$ 
is $\iota(\gamma)$ with the tag in a different position.

Thus, suppose we are given a loop $\gamma$ containing more than $2$ generators.  
We partition $\gamma$ into \emph{runs} of a single generator, and because $\gamma$ 
has at least $3$ generators, we can write, $\gamma = abcx$, where $a$, $b$, and $c$ 
are runs of distinct generators, and $x$ is the remainder of $\gamma$, which might be 
empty.  Abusing notation, we will write $a$, $b$, $c$ to denote 
a run of any size of the $a$, $b$, $c$ generators.
The trivalent fatgraph shown in Figure~\ref{fig:reduce_to_rank_2} 
has boundary of the form $abcx + Bc + bA + CaX$.  Note that this 
fatgraph $\emph{is}$ trivalent; it cannot fold by the assumption that 
$a$, $b$, $c$ are maximal runs of distinct generators.  Also, the 
edge lengths can be chosen so that all the loops have size $7$.

\begin{figure}[ht]
\labellist
\small\hair 2pt
 \pinlabel {$a$} at 34 74
 \pinlabel {$b$} at 55 34
 \pinlabel {$c$} at 26 11
 \pinlabel {$x$} at 10 48
 
 \pinlabel {$B$} at 70 34
 \pinlabel {$B$} at 81 60
 \pinlabel {$c$} at 82 22
 
 \pinlabel {$A$} at 43 88
 \pinlabel {$b$} at 81 76
 \pinlabel {$A$} at 72 97
 
 \pinlabel {$X$} at -5 50
 \pinlabel {$C$} at 15 -3
 \pinlabel {$C$} at 94 11
 \pinlabel {$a$} at 100 104
\endlabellist
\centering
\includegraphics[scale=1.0]{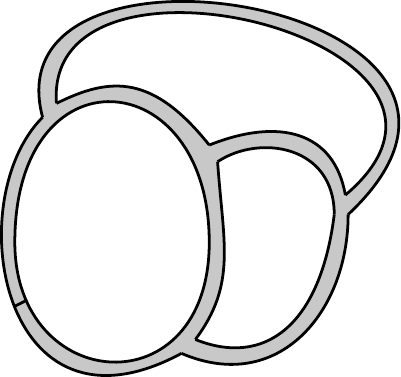}
\caption{A trivalent fatgraph showing how a loop of the form $abcx$ is fatgraph 
equivalent to two loops containing only two generators, plus a loop of the 
form $XcA$.}
\label{fig:reduce_to_rank_2}
\end{figure}
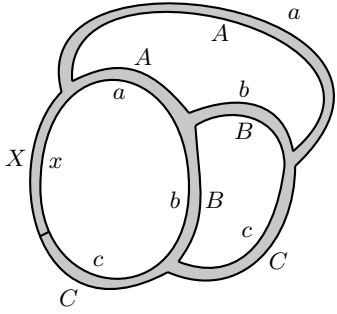

Therefore, $abcx$ is fatgraph equivalent to two loops with two 
generators, plus one loop of the form $Acx$.  The latter loop 
has, then, fewer runs, and we can repeat this procedure until $Acx$ contains 
only two generators.  At that point, we have shown that $\gamma$ 
is fatgraph equivalent to a collection of loops, each of which contains at 
most two generators. 
\end{proof}

Now we can complete the proof of Proposition~\ref{prop:higher_rank_trivalent}.
Given a vector $v \in \L'^+_0$, we may take a sufficient multiple to assume that 
$v$ is integral.  Let $S$ be the collection of loops represented by $v$.  
We've shown that $S$ is fatgraph equivalent to a collection of loops $S'$, 
where each loop in $S'$ contains at most two generators.  We can write 
$S' = \bigcup_{\{x,y\}} S'_{x,y}$, where each $S'_{x,y}$ contains the loops in $S'$ 
containing the generators $x$ and $y$.  There is an ambiguity about where to 
put loops which contain a single generator; place them arbitrarily, though we will rearrange 
them presently.  
Now, each collection $S'_{x,y}$ might not be homologically trivial.  
However, after multiplying the original vector $v$ by $7$, we may assume that the 
homological defect of each generator in each collection is a multiple of $7$.
Therefore, we can make each collection $S'_{x,y}$ homologically trivial by redistributing 
the loops consisting of a single generator (or, if necessary, introducing $\iota$ pairs).

Since each collection $S'_{x,y}$ is homologically trivial, we can apply 
Proposition~\ref{prop:rank_2_trivalent} to find a trivalent fatgraph with boundary 
$S'_{x,y} + \1_{x,y}$, where $\1_{x,y}$ denotes the uniform vector 
in the set of loops containing generators $x$ and $y$.  Taking the union of 
these fatgraphs over all $x$, $y$, yields a trivalent fatgraph 
which shows that $S' = \bigcup_{\{x,y\}}S'_{x,y}$ is fatgraph equivalent to a 
collection of $\iota$ pairs, and therefore fatgraph equivalent to the empty set.  
That is, the collection $S$ represented by our original vector $v$ is 
fatgraph equivalent to the empty set, which is to say that there is a 
trivalent fatgraph with boundary $S + n\1$, for $n$ sufficiently large.  
This completes the proof.
\end{proof}

We now give the proof in the rank $2$ case, or rather, a description 
of the computation which proves it.

\begin{proof}[Proof of Proposition~\ref{prop:rank_2_trivalent}]
We wish to show that the cone $C$ contains an open projective 
neighborhood of $\1$.  To start, we describe some necessary background about 
cones and linear programming.
Determining if a point lies in the interior of the cone on 
a set of vectors can be phrased as a linear programming problem by using the 
following lemma.

\begin{lemma}\label{lem:cone_interior}
Let $x,v_1, \ldots, v_k \in \R^n$.  If the $v_i$ span $\R^n$ and there is an expression
$x = \sum_it_iv_i$ with $t_i > 0$ for all $i$, then $x$ lies in the 
interior of the cone spanned by the $v_i$.
\end{lemma}
\begin{proof}
To show that $x$ lies in the interior, it suffices to show that 
$x$ is not contained in any supporting hyperplane.  Thus, let $H$ 
be any supporting hyperplane for the cone spanned by the $v_i$.  
Because the $v_i$ span $\R^n$, there is some $j$ so that $v_j$ is 
not in $H$.  We have expressed $x = \sum_it_iv_i$ with, in particular, $t_j>0$.  
Therefore, if we decompose $\R^n = H \oplus \textnormal{span}\{v_j\}$, 
and express $x$ in this decomposition, 
we will find that the coefficient of $v_j$ is not zero, so $x$ is 
not contained in $H$.
\end{proof}

Using Lemma~\ref{lem:cone_interior}, we observe that if we are given 
the $v_i$ as the columns of the matrix $M$, and $x$ is a column vector, 
then a feasible point $y$ for the problem $Ay = x$, $y\ge 1$ 
provides a certificate that $x$ is contained in the interior.  
Feasibility testing can be phrased as a linear programming 
problem by setting the objective function to zero.  We remark that the lower bound 
on $y$ is arbitrary; if $x$ is in the interior, the linear program will 
succeed for \emph{some} lower bound, but we don't know a priori what it is.  

The proof of Proposition~\ref{prop:rank_2_trivalent} therefore reduces 
to the following computation.
\begin{enumerate}
\item Find a collection of vectors $V$ in the cone $C$ which together 
span the space $\L_0$.
\item Show that the uniform vector $\1$ lies in the cone on $V$.
\end{enumerate}
To find $V$, we simply tried many random vectors in $\L_0$ and 
checked if they were contained in $C$ by checking if they 
bounded a trivalent fatgraph.
Both steps (1) and (2) require linear programming:
in order to check that a 
vector bounds a trivalent fatgraph, we solve a linear programming problem 
derived from the \texttt{scallop} algorithm (\cite{scallop}), 
and to show that the uniform vector 
lies in the cone on $V$, we solve the linear programming problem 
derived from Lemma~\ref{lem:cone_interior}.  

In order to make the many linear programming problems in step (1) feasible, 
we need to choose \emph{low-density} vectors; that is, vectors with a 
small number of nonzeros.  Recall there are $4376$ short loops of 
length $7$ and rank $2$, and the space of homologically trivial 
linear combinations has dimension $4374$.  
We found a collection of $9626$ vectors, each with $8$ nonzeros, 
which span this $4374$-dimensional space.  This required solving 
a few tens of thousands of small linear programming problems (i.e. runs of \texttt{scallop}), 
which was easily accomplished using the linear programming backend 
\texttt{GLPK}\cite{GLPK}.
As we built this collection, we occasionally ran a much larger 
linear program to determine if the uniform vector was contained 
in the cone (not necessarily in the interior, as that is more difficult to 
solve in practice).  Once our cone did contain the uniform vector, 
we ran one final linear program to verify that it lay in the interior.

Even though these latter linear programs had only a few thousand columns 
and a few thousand rows, they proved quite difficult in practice.  
Fortunately, the proprietary software package 
\texttt{Gurobi}\cite{gurobi}, which offers a free academic license, 
was able to solve them in a few minutes.  We used \texttt{Sage}\cite{sage} 
to facilitate many of the final steps.
\end{proof}

\begin{remark}
It is important to highlight that the initial steps of the proof of the 
random $f$-folded surface theorem reduce the problem 
of finding a folded surface whose boundary is a given random loop to the 
problem of showing that a collection of tagged loops of 
a \emph{uniformly bounded} size ($7$) bounds a folded fatgraph, provided 
this collection is sufficiently close to uniform.  Proposition~\ref{prop:higher_rank_trivalent} 
shows that, indeed, a collection of tagged loops of size $7$ 
sufficiently close to uniform does bound a folded fatgraph.  We emphasize 
that this linear programming is done \emph{once} to solve this 
latter, uniformly bounded, one-time problem.
\end{remark}

This completes the proof of Theorem~\ref{thm:random_f_folded_theorem}.

\subsection{Sapir's group}\label{subsection:Sapir_group}

\begin{definition}
We define {\em Sapir's group} $C$ to be the HNN extension of $F_2:=\langle a,b\rangle$ by the
endomorphism $\phi:a \to ab, b \to ba$.
\end{definition}

In \cite{Sapir}, Problem.~8.1, Sapir posed explicitly the problem of determining whether $C$
contains a closed surface subgroup, and in fact conjectured (in private communication) that
the answer should be negative. This group was also studied by Crisp--Sageev--Sapir and
(independently) Feighn, who sought to find a surface subgroup or show that one did not exist.
Because of the attention this particular question has attracted, we consider it significant
that our techniques are sufficiently powerful to give a positive answer:

\begin{theorem}
Sapir's group $C$ contains a closed surface subgroup of genus 28.
\end{theorem}
\begin{proof}
The theorem is proved by exhibiting an explicit $f$-folded surface. Figure~\ref{sapir_example}
indicates a fatgraph whose fattening has four boundary components, three of which are (conjugates
of) $babaBABA$ and the fourth of which is $\phi^4(babaBABA)^{-3}$. The blue circles mark
the $babaBABA$ components. By taking a 3-fold cover we obtain a fatgraph whose fattening has
six boundary components, three of which are conjugates of $(babaBABA)^3$, and three of which are
conjugates of $\phi^4(babaBABA)^{-3}$. In the HNN extension $F*_\phi$ we can glue these boundary
components in pairs, giving a closed surface $S$ together with a map $\pi_1(S) \to F*_\phi$.

The surface is $f$-folded, and therefore the resulting map of the surface group is injective. To
see this, note that the $babaBABA$ components are disjoint from each other,
the underlying fatgraph is Stallings folded, and the
$f$-vertices (indicated in red) are all 2-valent. 
\begin{figure}[htpb]
\centering
\includegraphics[scale=0.3]{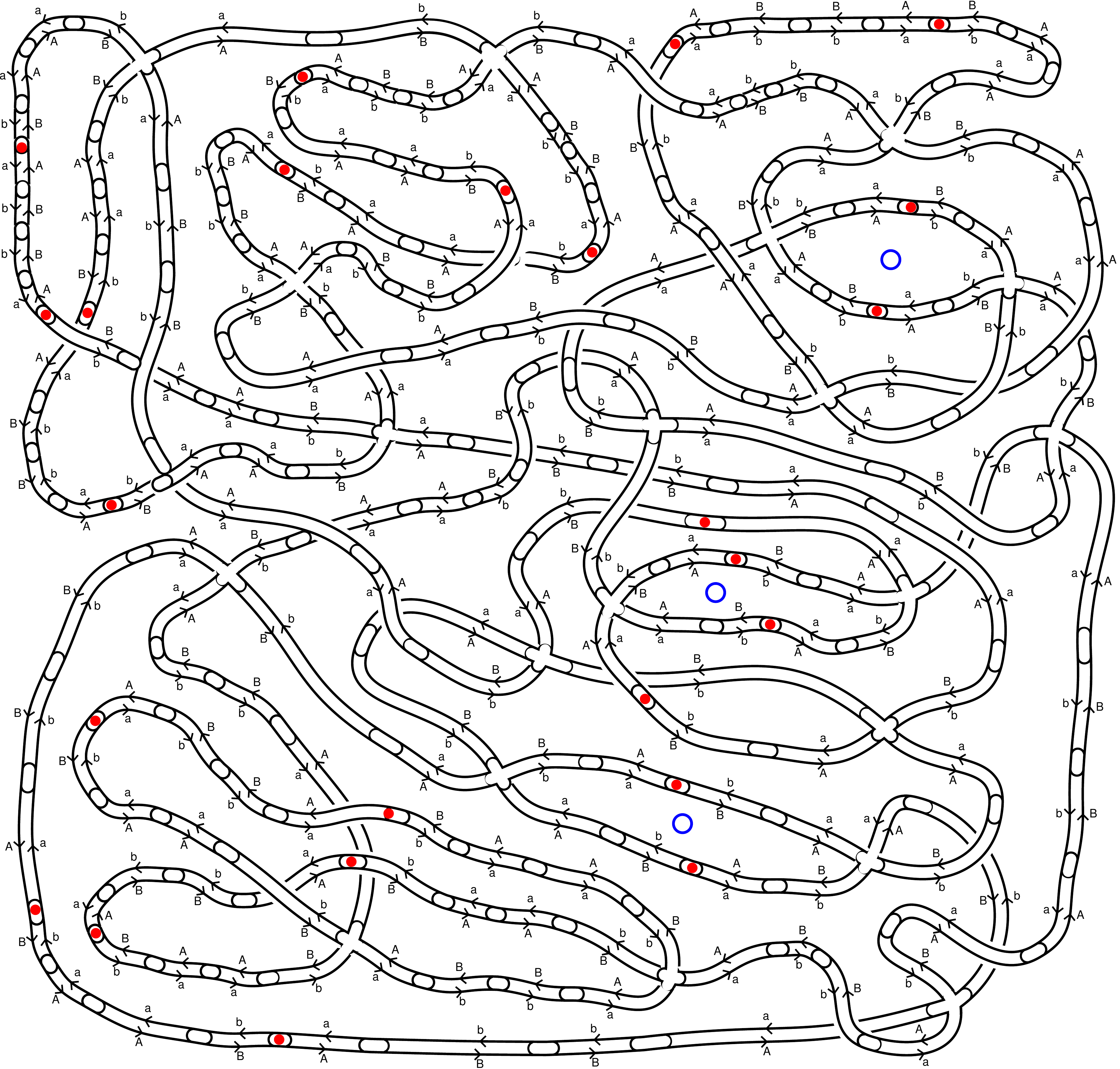}
\caption{A fatgraph bounding $3\cdot babaBABA+\phi^4(babaBABA)^{-3}$}\label{sapir_example}
\end{figure}
\end{proof}

\begin{remark}
In fact, Sapir expressed the opinion that ``most'' ascending
HNN extensions of free groups should not contain surface subgroups, which is contradicted by
the Random $f$-folded Surface Theorem~\ref{thm:random_f_folded_theorem}. On the other hand,
the probabilistic estimates involved in the proof of this theorem are only relevant for
endomorphisms taking generators to very long words, and therefore Sapir's group seems to be
an excellent test case.
\end{remark}

\end{document}